\documentclass[11pt,a4paper,reqno]{amsart} %maths
\usepackage{amssymb,graphicx}%maths
\newtheorem{theo}{Theorem}[section] 
\newtheorem{prop}[theo]{Proposition}  
\newtheorem{lemma}[theo]{Lemma}

\newcounter{mnotecount}[section]

\renewcommand{\themnotecount}{\thesection.\arabic{mnotecount}}

\newcommand{\mnote}[1]%{}%
{\protect{\stepcounter{mnotecount}}$^{\mbox{\footnotesize
$%\!\!\!\!\!\!\,
\bullet$\themnotecount}}$ \marginpar{%\color{red}%
\raggedright\tiny\em
$\!\!\!\!\!\!\,\bullet$\themnotecount: #1} }

\newcommand{\PP}{\mathbb{P}}

\newcommand{\Rho}{\mathrm{P}}
\def\p{\partial}
\def\be{\begin{equation}}

\def\ee{\end{equation}}

\def\bea{\begin{eqnarray}}
\def\eea{\end{eqnarray}}

\def\T{\mbox{Tr}}

\def\ov{\overline}

\newcommand{\intprod}{\;\rule{4pt}{.3pt}\rule{.3pt}{5pt}\;}
\newcommand{\xoo}[3]{\begin{picture}(40,11)
\put(4,1.5){\makebox(0,0){$\times$}}
\put(20,1.5){\makebox(0,0){$\bullet$}}
\put(36,1.5){\makebox(0,0){$\bullet$}}
\put(4,1.5){\line(1,0){32}}
\put(4,8){\makebox(0,0){$\scriptstyle #1$}}
\put(20,8){\makebox(0,0){$\scriptstyle #2$}}
\put(36,8){\makebox(0,0){$\scriptstyle #3$}}
\end{picture}}
\newcommand{\oo}[2]{\begin{picture}(24,11)
\put(4,1.5){\makebox(0,0){$\bullet$}}
\put(20,1.5){\makebox(0,0){$\bullet$}}
\put(4,1.5){\line(1,0){16}}
\put(4,8){\makebox(0,0){$\scriptstyle #1$}}
\put(20,8){\makebox(0,0){$\scriptstyle #2$}}
\end{picture}}
\newcommand{\Aone}[1]{\begin{picture}(8,11)
\put(4,1.5){\makebox(0,0){$\bullet$}}
\put(4,8){\makebox(0,0){$\scriptstyle #1$}}
\end{picture}}
\begin{document}
% \date{August 10, 2014}
\def\thefootnote{}\footnotetext{\phantom{August 10, 2014}}
%%%%%%%%%%%%%%%%%%%%%%%%%%%%%%%%%%%%%%%%
\title{Metrisability of three-dimensional path geometries}
%%%%%%%%%%%%%%%%%%%%%%%%%%%%%%%%%%%%%%%%
\author{Maciej Dunajski}
\address{Department of Applied Mathematics and Theoretical Physics\\ 
University of Cambridge\\ Wilberforce Road, Cambridge CB3 0WA, UK.}
\email{m.dunajski@damtp.cam.ac.uk}
\author{Michael Eastwood}
\address{School of Mathematical Sciences\\
Australian National University\\
ACT 0200, Australia.}
\email{meastwoo@member.ams.org}

\begin{abstract}
Given a projective structure on a three-dimensional manifold, we find explicit
obstructions to the local existence of a Levi-Civita connection in the
projective class. These obstructions are given by projectively invariant
tensors algebraically constructed from the projective Weyl curvature. We show,
by examples, that their vanishing is necessary but not sufficient for local
metrisability.
\end{abstract}
%%%%%%%%%%%%%%%%%%%%%%%%%%%%%%%%%%%%%%%%%%
%PACS numbers:
%%%%%%%%%%%%%%%%%%%%%%%%%%%%%%%%%%%%%%%%%%

\maketitle

\section{Introduction}
There are several inequivalent geometric structures that give rise to a
preferred family of curves on a smooth $n$-manifold~$M$. A {\em path
geometry\/} on $M$ is a locally defined family of unparametrised smooth curves
(called {\em paths\/}), one through each point and in each direction. A path
geometry is {\em projective} if its paths are the unparametrised geodesics of a
torsion-free connection $\nabla$ on~$TM$. The corresponding projective
structure $(M, [\nabla])$ is then defined by the equivalence class of
torsion-free connections sharing their unparametrised geodesics with~$\nabla$.
Finally a path geometry is {\em metrisable\/} if its paths are the
unparametrised geodesics of a (pseudo-)Riemannian metric~$g$. In this case the
path geometry is, of course, also projective since the underlying projective
structure is defined by the Levi-Civita connection of $g$. The converse does
not hold: a general projective structure does not contain a Levi-Civita
connection of any metric.

The characterisation of metrisable projective structures is a classical
problem, which goes back to the work of Roger Liouville~\cite{Liouville}. This
problem has recently been solved if $n=2$: in this case (assuming
real-analyticity for sufficiency, as one must) necessary and sufficient
conditions are given by the vanishing of a set of projective differential
invariants~\cite{BDE}, the simplest of which is of differential order five in
the connection coefficients of a chosen $\nabla\in[\nabla]$. The case $n=2$ is
special---the projective Weyl curvature vanishes on a surface. This is
no longer the case if $n=3$, where the first set of obstructions already arises
at order one, and is algebraic in the projective Weyl tensor. In this paper we
shall present some of these obstructions as explicit projectively invariant
tensors constructed algebraically from the Weyl curvature. 

As is often done in differential geometry, we shall adorn tensors with indices
in order to denote the type of the tensor. Thus, we may denote a vector or
vector field by $X^a$ but $\omega_a$ always denotes a co-vector or $1$-form.
The canonical pairing between vectors and co-vectors is denoted by repeating an
index so that $X^a\omega_a$ is the scalar that would often be written without
indices as $X\intprod\omega$. For any tensor $\psi_{abc}$, we shall denote its
skew part by $\psi_{[abc]}$ and its symmetric part by $\psi_{(abc)}$. For
example, if $\omega_{ab}$ is a $2$-form, then
$$\nabla_{[a}\omega_{bc]}\quad\mbox{and}\quad
X^a\nabla_a\omega_{bc}-2(\nabla_{[b}X^a)\omega_{c]a},$$
for any torsion-free connection~$\nabla_a$, are the exterior derivative and the
Lie derivative in the direction of the vector field~$X^a$, respectively. Such
formulae are not meant to imply any choice of local co\"ordinates. More
precisely, this is Penrose's {\em abstract index notation\/}~\cite{OT}. 

Projective structures are reviewed at the start of Section~\ref{section_1}.
Here, suffice it to say that the primary invariant of a projective structure is
the projective Weyl tensor $W_{ab}{}^c{}_d$, an irreducible part of the
curvature of any connection in the projective class $[\nabla]$ satisfying
$$W_{ab}{}^c{}_d=W_{[ab]}{}^c{}_d,\quad W_{[ab}{}^c{}_{d]}=0,\quad
W_{ab}{}^a{}_d=0.$$
This article is concerned exclusively with the $3$-dimensional case. To
formulate one of our results define a traceless tensor
${V^{ab}}_c={V^{(ab)}}_c$ in terms of the projective Weyl curvature
$W_{ab}{}^c{}_d$ and an arbitrarily chosen non-degenerate section
$\epsilon^{abc}$ of $\Lambda^3(TM)$ by 
\be
\label{V_from_W}
{V^{ab}}_c=\epsilon^{dea}W_{de}{}^b{}_c.
\ee
\begin{theo}
\label{theorem_1}
Let $A,B,C,D,F,J,K,L$ be the symmetric tensors defined by
\begin{align}
\label{ABCD}
A^{ab}&=\textstyle\bigodot(V^{ap}{}_qV^{bq}{}_p),\quad
B^{abc}=\textstyle\bigodot(V^{ap}{}_qV^{bq}{}_rV^{cr}{}_p),\quad
C^{abc}=\textstyle\bigodot(V^{ab}{}_pV^{pq}{}_rV^{cr}{}_q),\nonumber\\
D^{abcd}&=\textstyle\bigodot(V^{ab}{}_pV^{pq}{}_rV^{cr}{}_sV^{ds}{}_q),\quad
J^{abcdef}=\textstyle\bigodot
(V^{ab}{}_pV^{cd}{}_qV^{pq}{}_rV^{er}{}_sV^{st}{}_uV^{fu}{}_t),
\nonumber\\
F^{abcd}&=\textstyle\bigodot(V^{ab}{}_pV^{cd}{}_qV^{pr}{}_sV^{qs}{}_r),
\quad
K^{abcdef}=\textstyle\bigodot
(V^{ab}{}_pV^{cd}{}_qV^{ep}{}_rV^{fq}{}_sV^{rt}{}_uV^{su}{}_t)
\nonumber\\
L^{abcdef}&=\textstyle\bigodot
(V^{ab}{}_pV^{cd}{}_qV^{pq}{}_rV^{rs}{}_tV^{et}{}_uV^{fu}{}_s),
\end{align}
where $\bigodot$ denotes symmetrisation over the non-contracted indices,
and let 
\begin{equation}\label{the_determinant}
T= 24J+12K-24L-24B\odot B-24C\odot C+40B\odot C-24A\odot D+6A\odot F.
\end{equation}
In general, the tensor $T^{abcdef}$ does not vanish. However, if\/ $\nabla$ is
projectively equivalent to a Levi-Civita connection then $T^{abcdef}\equiv 0$.
\end{theo}
\begin{theo}\label{theorem_2}
With tensors 
$A^{ab},B^{abc},C^{abc},D^{abcd},F^{abcd},J^{abcdef},K^{abcdef},L^{abcdef}$
defined as in Theorem~\ref{theorem_1}, the following
\begin{equation}\begin{array}{l}
\label{MN}
C^{abc}-2B^{abc},\quad
F^{abcd}-2D^{abcd},\quad
J^{abcdef}-2L^{abcdef},\\
3J^{abcdef}-2C^{(abc}C^{def)},\quad
J^{abcdef}-4K^{abcdef}+4A^{(ab}D^{cdef)}
\end{array}\end{equation}
are generally non-zero but vanish if\/ $\nabla$ is projectively equivalent to a
Levi-Civita connection.
\end{theo}
In fact, since 
$$\begin{array}{rcl}T&=&
12(J-2L)-3(J-4K+4A\odot D)+5(3J-2C\odot C)\\[3pt]
&&\quad{}+2(C-2B)\odot(6B-7C)+6A\odot (F-2D),\end{array}$$
Theorem~\ref{theorem_2} implies Theorem~\ref{theorem_1}. However, we shall see
that the invariant $T$ arises in a more fundamental way, already described
in~\cite{BDE}. More specifically, one can construct from the Weyl curvature a
homomorphism $\Xi:\bigodot^2(TM)\to\bigodot^3(TM)$ that must have a non-trivial
kernel in the metrisable case and $T^{abcdef}$ is defined to be what amounts to
the determinant of this homomorphism: $T^{abcdef}X_aX_bX_cX_dX_eX_f$ is
characterised as being the determinant of the composition
$$\textstyle\bigodot^2(TM)\xrightarrow{\;\Xi\;}\bigodot^3(TM)
\xrightarrow{\;X\intprod\underbar{\enskip}\;}
\bigodot^2(TM)$$
for any $1$-form~$X_a$. The homomorphism $\Xi$ is constructed by forming and
prolonging the metrisability equation (e.g.~\cite{East_Mat}) and is the natural
first step~\cite{nurowski} in searching for a metric in a given projective
class. We should point out that these constructions are carried out having
arbitrarily chosen a non-zero section $\epsilon^{abc}$ of the line bundle
$\Lambda^3(TM)$. However, since a different section only changes the scale of
the various obstruction tensors at each point, whether they vanish or not is
unaffected. In the main body of this article we shall restore precision by
introducing {\em projective weights\/}, in effect a mechanism for keeping track
of the scale of~$\epsilon^{abc}$.

Theorem~\ref{theorem_2} is established by a different route, which seemingly
creates a proliferation of projectively invariant obstructions to
metrisability. For example, if one considers tensors such as $T^{abcdef}$,
taking values in $\bigodot^6(TM)$ and of degree $6$ in the projective Weyl
tensor, equivalently in the tensor $V^{ab}{}_c$, then one finds an
$8$-dimensional space of obstructions in the $11$-dimensional space of
projective invariants of this type. Indeed, there is already a quadratic
obstruction as follows. 
\begin{theo}\label{theorem_3} In order that a
projective structure\/ $[\nabla]$ be metrisable, the invariant tensor
\begin{equation}\label{this_is_Q}
Q_{ab}{}^c=\epsilon_{pq(a}V^{pr}{}_{b)}V^{qc}{}_r\end{equation}
must vanish, whilst in general it does not.
\end{theo}
\noindent And if one prefers a scalar obstruction, then there is one of 
degree~$3$ as follows.
\begin{theo}\label{theorem_4} In order that a projective structure\/ $[\nabla]$
be metrisable, the invariant scalar 
$$S=\epsilon_{abc}V^{ap}{}_qV^{bq}{}_rV^{cr}{}_p$$
must vanish, whilst in general it does not.
\end{theo}
Besides proving these theorems, we shall provide a systematic way of creating
many more invariants. Nevertheless, we shall show by examples,
in~\S\ref{egorov_subsection}--\ref{newtonian_subsection}, that this
proliferation of invariants is insufficient to characterise the metrisable
projective structures. In Section~\ref{section_4} we shall reformulate the
problem for path geometries in terms of systems of two second order ODE for the
unparametrised paths.

We would like to thank Katharina Neusser for pointing out the projectively 
invariant pairing (\ref{Neusser_pairing}) as a useful device in understanding 
the metrisability equation and also for helpful discussions concerning the 
form of the tensor $V^{ab}{}_c$ for the Egorov and Newtonian structures
in~\S\ref{egorov_subsection}--\ref{newtonian_subsection}.

Finally, as detailed in~\S\ref{egorov_subsection}, we would like to thank
Vladimir Matveev for drawing our attention to an alternative
proof~\cite{BSK_VSM} of the non-metrisability of the Egorov structure and also
for his pertinent comments concerning the possible values of the degree of
mobility of Riemannian and Lorentzian metrics in $3$ dimensions.

\section{Projective structures and metrisability}
\label{section_1}
Let $M$ be a smooth manifold. Let us consider an equivalence
class $[\nabla]$ of torsion-free connections on~$TM$, where to say that $\nabla$
and $\hat{\nabla}$ belong to $[\nabla]$ is to say that there is a $1$-form
$\Upsilon_a$ such that 
\be
\label{equivalence}
\hat{\nabla}_a X^b=\nabla_a X^b+\Upsilon_a X^b+{\delta_a}^b\Upsilon_c X^c. 
\ee
This is the condition for the geodesics sprays of $\nabla$ and $\hat{\nabla}$
on $TM$ to have the same projection to~$\PP(TM)$. Therefore, it is exactly the
condition that all connections in $[\nabla]$ share the same unparametrised
geodesics on~$M$. In other words, the equivalence class $[\nabla]$
operationally defines what is a projective structure.

The curvature of a given connection $\nabla\in[\nabla]$ is defined by
\[
[\nabla_a, \nabla_b]X^c={{R_{ab}}^c}_d X^d,
\]
and can be uniquely decomposed as
\be
\label{defofRho}
{{R_{ab}}^c}_d= {{W_{ab}}^c}_d+
\delta_a{}^c\Rho_{bd}-\delta_b{}^c\Rho_{ad} +\beta_{ab}\delta_d{}^c
\ee
where $\beta_{ab}=-2\Rho_{[ab]}$ and ${{W_{ab}}^c}_d$ is totally trace-free.
Then $\Rho_{ab}$ is the {\em projective Schouten tensor\/} and $W_{ab}{}^c{}_d$
is the {\em projective Weyl tensor\/}. If we change connection in the
projective class using (\ref{equivalence}) then
\[
\hat{\Rho}_{ab}={\Rho}_{ab}-\nabla_a\Upsilon_b+\Upsilon_a\Upsilon_b,
\quad
\hat{\beta}_{ab}=\beta_{ab}+2\nabla_{[a}\Upsilon_{b]},
\]
whilst ${{{W_{ab}}^{c}}}_d$ remains unchanged. 

We now specialise to the case when $M$ is $3$-dimensional. Computing the 
effect of a change of connection (\ref{equivalence}) on a $3$-form 
$\eta_{abc}$, we find that 
\begin{equation}\label{change_on_three-forms}
\hat\nabla_a\eta_{bcd}=\nabla_a\eta_{bcd}-3\Upsilon_a\eta_{bcd}
-\Upsilon_b\eta_{acd}-\Upsilon_c\eta_{bad}-\Upsilon_d\eta_{bca}
=\nabla_a\eta_{bcd}-4\Upsilon_a\eta_{bcd}\end{equation} 
and so if $M$ is oriented (as we shall suppose henceforth) and $\eta_{abc}$ is
chosen to be everywhere non-vanishing (we say that $\eta_{abc}$ is a 
{\em choice of scale\/}), then we may specify $\Upsilon_a$ by requiring that
$\nabla_a\eta_{bcd}-4\Upsilon_a\eta_{bcd}=0$, thus obtaining a unique
connection $\hat\nabla$ in the projective class such that
$\hat\nabla_a\eta_{bcd}=0$. We shall refer to the connections obtained in this
way as {\em special\/}. {From} (\ref{defofRho}), we find that
$$[\nabla_a,\nabla_b]\eta_{cde}=-4\beta_{ab}\eta_{cde}$$
and conclude that $\beta_{ab}=0$ for special connections and hence that the
Schouten tensor $\Rho_{ab}$ is symmetric. (If $M$ is not oriented, then we
define a choice of scale to be a nowhere-vanishing section of $(\Lambda_M^3)^2$
instead, where $\Lambda_M^3$ is the bundle of $3$-forms on~$M$.) For any
$w\in{\mathbb{R}}$, it is convenient to denote by ${\mathcal{E}}(w)$ the line
bundle $(\Lambda_M^3)^{-w/4}$, invariantly defined as the bundle whose fibre at
$p\in M$ is the $1$-dimensional vector space
$$\{\phi:\Lambda_+^3T_p^*M\to{\mathbb{R}}\mid 
\phi(\lambda\omega)=\lambda^{w/4}\phi(\omega),\;\forall\lambda>0\},$$ 
where $\Lambda_+^3T_p^*M$ denotes the $3$-forms at $p$ positive with respect to
the orientation, and we shall refer to a section $\rho$ of ${\mathcal{E}}(w)$
as a {\em projective density of weight $w$\/}. There are canonical isomorphisms
${\mathcal{E}}(k)={\mathcal{E}}(1)^{\otimes k}$ for $k\in{\mathbb{Z}}$. Also,
by construction, there is an identification ${\mathcal{E}}(-4)=\Lambda_M^3$,
which we shall write as $\rho\mapsto\rho\epsilon_{abc}$ for $\rho$ of
projective weight~$-4$. Equivalently, we have a canonical volume form
$\epsilon_{abc}$ of weight $4$, that is to say a canonical section
of~$\Lambda_M^3(4)$, the tensor product $\Lambda_M^3\otimes{\mathcal{E}}(4)$.
Having done this, a scale may be alternatively specified as a nowhere-vanishing
density $\sigma$ of projective weight $1$, so that
$\eta_{abc}=\sigma^{-4}\epsilon_{abc}$ is the corresponding volume form. This
is the viewpoint we shall adopt henceforth. In summary, we are working with
special connections specified by a choice of projective density $\sigma$ of
weight~$1$. Choosing a different scale, say $\hat\sigma=\Omega^{-1}\sigma$ for
some nowhere-vanishing function~$\Omega$, induces a projective change of
connection (\ref{equivalence}) where $\Upsilon=\Omega^{-1}d\Omega$. In the
presence of a scale $\sigma\in\Gamma(M,{\mathcal{E}}(1))$, we may view a
projective density $\rho\in\Gamma(M,{\mathcal{E}}(w))$ as a smooth function,
specifically $f=\rho/\sigma^w$, but if we change scale
$\sigma\mapsto\hat\sigma=\Omega^{-1}\sigma$, then this function changes
according to $\hat{f}=\Omega^wf$. Finally, for any scale
$\sigma\in\Gamma(M,{\mathcal{E}}(1))$, the line bundles ${\mathcal{E}}(w)$
inherit connections characterised by $\nabla\sigma^w=0$ and then, if $\rho$ is
a projective density of weight $w$, we see that
$$\hat\nabla_a\rho=\nabla_a\rho+w\Upsilon_a\rho$$
and note that this is consistent with (\ref{change_on_three-forms}) and the
identification $\Lambda_M^3={\mathcal{E}}(-4)$. Otherwise said, the
tautological $3$-form $\epsilon_{abc}$ of weight $4$ is covariant constant
($\nabla_a\epsilon_{bcd}=0$) for any special connection on $\Lambda_M^3(4)$. We
shall denote by $\epsilon^{abc}$ the induced canonical section of the dual
bundle $\Lambda^3(TM)(-4)$ normalised so that
$$\epsilon_{abc}\epsilon^{def}=\delta_a{}^d\delta_b{}^e\delta_c{}^f
+\delta_a{}^e\delta_b{}^f\delta_c{}^d+\delta_a{}^f\delta_b{}^d\delta_c{}^e
-\delta_a{}^e\delta_b{}^d\delta_c{}^f
-\delta_a{}^d\delta_b{}^f\delta_c{}^e
-\delta_a{}^f\delta_b{}^e\delta_c{}^d,$$
equivalently that $\epsilon^{abc}\epsilon_{abc}=6$.

For those who find this discussion arcane, we should admit that a low-tech
alternative is to regard formul{\ae} containing $\epsilon^{abc}$ as defined
using an arbitrarily chosen nowhere-vanishing section $\epsilon^{abc}$ of
$\Lambda^3(TM)$ and that the projective weight simply keeps track of how these
expressions change if $\epsilon^{abc}$ is rescaled. Looking back at
(\ref{V_from_W}) now, we see that $V^{ab}{}_c$ may be invariantly regarded as 
a tensor satisfying
$$V^{ab}{}_c=V^{(ab)}{}_c,\quad V^{ab}{}_a=0,\quad
\mbox{and of projective weight~$-4$}.$$
Similarly, the scalar invariant $S$ from Theorem~\ref{theorem_4} is a
projective density of weight~$-8$.

Even more arcane, yet correspondingly even more useful, is to record both the
symmetries of a tensor and its projective weight by decorating a suitable
Dynkin diagram in the style of~\cite{beastwood}. An explanation sufficient for
our purposes is as follows. We shall denote by \xoo{k}{l}{m} the bundle of 
totally trace-free tensors
$$S^{\overbrace{\scriptstyle ab\cdots c}^{\scriptstyle m
\makebox[0pt][l]{ \scriptsize contravariant indices}}}
{}_{\underbrace{\scriptstyle de\cdots f}_{\scriptstyle l
\makebox[0pt][l]{ \scriptsize covariant indices}}}\hspace{6em}
\mbox{of projective weight }k+2l-m$$
and symmetric in its covariant and contravariant indices. These constitute a
complete list of the irreducible bundles on an oriented projective
$3$-manifold as a parabolic geometry in the sense of~\cite{parabook}. It 
is also useful to introduce the {\em degree\/} of such a bundle as
\begin{equation}\label{def_of_degree}\deg(\xoo{k}{l}{m})=-(3k+2l+m)/4.
\end{equation}
(It is the action of the {\em grading element\/} normalised as
in~\cite{parabook}.) The point about the degree is that it simply adds under
tensor product, for example, 
\begin{equation}\label{tensor_product}
\textstyle\Lambda_M^1\otimes\Lambda_M^1=\xoo{-2}{1}{0}\otimes\xoo{-2}{1}{0}=
\xoo{-4}{2}{0}\oplus\xoo{-3}{0}{1}=\bigodot^2\!T^*M\oplus\Lambda_M^2
\end{equation}
and changes sign when taking duals, for example,
$$\deg(TM)=\deg(\xoo{1}{0}{1})=-1\quad\mbox{whilst}\quad
\deg(\Lambda_M^1)=\deg(\xoo{-2}{1}{0})=1.$$
Also notice that our previous discussion concerning projective weights and the
tautologically defined tensors $\epsilon_{abc}$ and $\epsilon^{abc}$ is
implicitly incorporated into this notation. For example, the identification
$\Lambda_M^2=\xoo{-3}{0}{1}$ in (\ref{tensor_product}) is given by
$\omega_{ab}\mapsto\epsilon^{abc}\omega_{bc}$. Fundamental for this article
is~(\ref{V_from_W}), giving $V^{ab}{}_c\in\Gamma(M,\xoo{-4}{1}{2})$ of
projective weight~$-4$. This irreducible tensor is every bit as good as the
unweighted projectively invariant Weyl tensor~$W_{ab}{}^c{}_d$, the inverse to
(\ref{V_from_W}) being given by
$W_{ab}{}^c{}_d=\frac12\epsilon_{eab}V^{ec}{}_d$.

\subsection{Metrisability} 
As already remarked, a (pseudo-)Riemannian metric $g$ on $M$ gives rise to a
projective structure~$[\nabla]$, namely the one that contains the Levi-Civita
connection of~$g$. Hence we obtain a first order non-linear operator
\be\textstyle
\label{1st_order_op}
J^1(\bigodot^2\!\Lambda_M^1)\supset 
J^1(\bigodot_{\mathrm{nd}}^2\Lambda_M^1)\xrightarrow{\,\sigma^0\,}
{\mathrm{Pr}}(M),
\ee
which carries a metric to its associated projective structure, where
\begin{itemize}\addtolength{\itemsep}{3pt}
\item $\bigodot^2\!\Lambda_M^1$ is the vector bundle of symmetric covariant 
tensors,
\item $\bigodot_{\mathrm{nd}}^2\Lambda_M^1$ is the subbundle of non-degenerate
such tensors,
\item $J^1(\bigodot^2\!\Lambda_M^1)$ and 
$J^1(\bigodot_{\mathrm{nd}}^2\Lambda_M^1)$ are their first jet bundles 
(e.g.~\cite{S}), 
\item ${\mathrm{Pr}}(M)$ is the affine bundle of projective structures on~$M$.
\end{itemize}
We note that, in the $3$-dimensional case,
\begin{itemize}
\item $\bigodot^2\!\Lambda_M^1$ is a vector bundle of rank~$6$,
\item $J^1(\bigodot^2\!\Lambda_M^1)$ is a vector bundle of rank~$24$,
\item $J^1(\bigodot_{\mathrm{nd}}^2\Lambda_M^1)$ is a Zariski-open subbundle,
\item ${\mathrm{Pr}}(M)$ is modelled on $\xoo{-3}{2}{1}$, which has rank~$15$.
\end{itemize}
Even taking into account that constant multiples of a given metric give rise to
the same connection and hence the same projective structure, the dimensions
indicate that $\sigma^0$ should be surjective and it is easy to check, using
local co\"ordinates, that this is, indeed, the case.
Differentiating~(\ref{1st_order_op}), however, gives rise to its first 
prolongation
$$\textstyle J^2(\bigodot_{\mathrm{nd}}^2\Lambda_M^1)\xrightarrow{\,\sigma^1\,}
J^1({\mathrm{Pr}}(M))$$
where the left hand side is now a Zariski-open subbundle of a vector bundle of
rank $60$ whilst the right hand side is an affine bundle modelled on
$J^1(\,\xoo{-3}{2}{1})$ a vector bundle also of rank~$60$. Taking scaling into
account, it follows that $\sigma^1$ cannot be surjective. In other words,
already at first order in the projective structure, we expect to see
obstructions to metrisability. Theorem \ref{theorem_1} shows that, indeed,
there are obstructions at this order which are algebraic in components of the
projective Weyl curvature. There is no restriction on the value of this
curvature as the following lemma shows.
\begin{lemma}\label{all_weyl_tensors_arise}
Choose $n\geq 3$ and let $W_{ab}{}^c{}_d$ be any element of 
$({\mathbb{R}}^n)^*\otimes({\mathbb{R}}^n)^*\otimes{\mathbb{R}}^n
\otimes({\mathbb{R}}^n)^*$ satisfying
\begin{equation}\label{weyl_symmetries}
W_{ab}{}^c{}_d=W_{[ab]}{}^c{}_d,\quad W_{[ab}{}^c{}_{d]}=0,\quad
W_{ab}{}^a{}_d=0.\end{equation} 
Let $M$ be an $n$-dimensional manifold and let $p\in M$ be an arbitrarily
chosen point. Then there is a torsion-free connection on $TM$ whose projective
Weyl curvature at $p$ is $W_{ab}{}^c{}_d$.
\end{lemma}
\begin{proof} The construction need only be local since connections can be
patched together by a partition of unity. In local
co\"ordinates~$(x^1,x^2,\ldots,x^n)$, the connection
$$\nabla_bX^c=\frac{\partial X^c}{\partial x^b}
+\frac23x^aW_{a(b}{}^c{}_{d)}X^d$$
has projective Weyl curvature $W_{ab}{}^c{}_d$ at the origin.
\end{proof}
In fact, the proof is also valid when $n=2$ but the statement is vacuous since
only the zero tensor satisfies the symmetries~(\ref{weyl_symmetries}). But for
$n=3$ space of the tensors satisfying (\ref{weyl_symmetries}) is
$15$-dimensional. This is clear from the alternative encoding of the Weyl
curvature as the tensor $V^{ab}{}_c$ in (\ref{V_from_W}) satisfying
$V^{ab}{}_c=V^{(ab)}{}_c$ and $V^{ab}{}_a=0$. Alternatively, the dimension of
any finite-dimensional irreducible tensor bundle can be computed from its
highest weight and the Dynkin diagram notation established above is designed 
with this in mind. There are various algorithms and computer implementations 
thereof. For example
$${\mathrm{rank}\,}(\xoo{k}{l}{m})=\verb!dim([l,m],A2)!$$
where the right hand side of this equation is an instruction written in
LiE~\cite{LiE}. We shall soon use LiE more seriously. 

Henceforth, we shall use the terminology {\em metric\/} to mean 
(pseudo-)Riemannian metric. The signature plays no essential r\^ole in our 
considerations and can be discussed separately.

One can attack the metrisability problem directly, asking for a metric $g_{ab}$
such that its Levi-Civita connection be projectively equivalent to a given
connection. Although the resulting partial differential equations on
$g_{ab}$ are projectively invariant by construction, they are also non-linear.
A surprising observation, essentially due to Liouville~\cite{Liouville}, is
that there is a non-linear change of variables that turns this system into a
linear one. For the convenience of the reader, we summarise the conclusions in
$3$ dimensions here and refer to \cite{East_Mat,M} for detail.
\begin{theo}\label{metrisability}
There is a projectively invariant linear differential operator
$$\textstyle\bigodot^2(TM)(-2)
=\xoo{0}{0}{2}\to\xoo{-2}{1}{2}\quad\mbox{given by}\quad\sigma^{bc}\mapsto
(\nabla_a\sigma^{bc})_\circ\equiv
\textstyle\nabla_a\sigma^{bc}-\frac12\delta_a{}^{(b}\nabla_d\sigma^{c)d}$$
called the\/ {\em metrisability operator} and a projectively invariant
differential pairing
\begin{equation}\label{Neusser_pairing}
\xoo{0}{0}{2}\times\xoo{2}{0}{0}\to\xoo{1}{0}{1}\quad\mbox{given by}\quad
(\sigma^{ab},\tau)\mapsto
\textstyle\sigma^{ab}\nabla_b\tau-\frac12\tau\nabla_b\sigma^{ab}.\end{equation}
If $\sigma^{bc}$ is symmetric and of projective weight~$-2$,
i.e.~$\sigma^{bc}\in\Gamma(\xoo{0}{0}{2})$, then
$$\textstyle\det\sigma
\equiv\frac16\sigma^{ad}\sigma^{be}\sigma^{cf}\epsilon_{abc}\epsilon_{def}$$
is a projective density of weight~$2$,
i.e.~$\det\sigma\in\Gamma(\xoo{2}{0}{0})$. If a tensor 
$\sigma^{bc}\in\Gamma(\xoo{0}{0}{2})$ 
satisfies the projectively invariant\/ {\em metrisability equation} 
\begin{equation}\label{metrisability_equation}
(\nabla_a\sigma^{bc})_\circ=0,\end{equation}
then the pairing with its determinant vanishes:
\begin{equation}\label{key}\textstyle
\sigma^{ab}\nabla_b(\det\sigma)-\frac12(\det\sigma)\nabla_b\sigma^{ab}=0.
\end{equation}
Furthermore, wherever $\det\sigma$ is non-zero, the weight zero tensor
$g^{ab}\equiv(\det\sigma)\sigma^{ab}$ defines a metric whose Levi-Civita
connection lies in the given projective class. Finally, up to sign, all metrics
in a given projective class arise in this manner.
\end{theo}
\begin{proof}
As set forth in the statement of this theorem, these claims are
straightforwardly verified from the definitions, the only further observation
required being that (\ref{key}) can be rewritten on $\{\det\sigma\not=0\}$ as
$\hat\nabla(g^{ab})=0$ where $\hat\nabla_a$ is projectively equivalent to
$\nabla_a$ according to (\ref{equivalence}) if we take
$\Upsilon_a=-\frac14g_{ab}(\det\sigma)\nabla_c\sigma^{bc}$, where $g_{ab}$ is
the inverse to $g^{ab}$. (We have taken the opportunity here, following a 
suggestion of Katharina Neusser, to streamline the
exposition in~\cite{East_Mat} by highlighting the r\^ole of~(\ref{key}).)
\end{proof}
An informal summary of Theorem~\ref{metrisability} is that the metrisability of
a given projective structure is controlled by the projectively invariant 
{\em metrisability equation\/}~(\ref{metrisability_equation}), there being a
$2$--$1$ correspondence between non-degenerate solutions of this equation and
{\em positive\/} metrics in the projective class (note that $\sigma^{ab}$ and 
$-\sigma^{ab}$ give rise to the same metric, that these metrics have positive 
determinant (we call them {\em positive\/}), and that 
conversely if $g^{ab}$ is such a metric, then
$$\textstyle\sigma^{ab}\equiv(\det g)^{-1/4}g^{ab}$$
solves~(\ref{metrisability_equation}).)
%%%%%%%%%%%%%%%%%%%%%%%%%%%%%%%%%%%%%%%%%%%%%%%%%%%%%%%%%%%%%%%%%%%%%%%%%%%%%%
%% PRIVATE NOTE                                                             %% 
%% In odd dimensions there is a local distinction between positive and      %%
%% negative metrics according to the sign of their determinant and this has %%
%% nothing to do with choice of orientation! Well, it depends precisely on  %%
%% the signature and in three dimensions it's those of type (+,+,+) or      %%
%% (+,-,-) that are positive (so not so surpising in the end).              %%
%%%%%%%%%%%%%%%%%%%%%%%%%%%%%%%%%%%%%%%%%%%%%%%%%%%%%%%%%%%%%%%%%%%%%%%%%%%%%%
\begin{theo}\label{prolongation_obstruction}
If $\sigma^{bc}$ solves the metrisability 
equation~\eqref{metrisability_equation}, then
\begin{equation}\label{curvature_constraint}
\textstyle W_{ab}{}^c{}_e\sigma^{de}+W_{ab}{}^d{}_e\sigma^{ce}
+\frac{2}{3}\delta_{[a}{}^cW_{b]}{}_e{}^d{}_f\sigma^{ef}
+\frac{2}{3}\delta_{[a}{}^dW_{b]}{}_e{}^c{}_f\sigma^{ef}=0.\end{equation}
\end{theo}
\begin{proof}
If we write (\ref{metrisability_equation}) as 
$$\nabla_b\sigma^{cd}=\delta_b{}^c\mu^d+\delta_b{}^d\mu^c$$
for some field~$\mu^a$, then 
\begin{equation}\label{start_to_prolong}
(\nabla_a\nabla_b-\nabla_b\nabla_a)\sigma^{cd}
+2\delta_{[a}{}^c\nabla_{b]}\mu^d+2\delta_{[a}{}^d\nabla_{b]}\mu^c=0
\end{equation}
but from (\ref{defofRho}) using, without loss of generality, a special
connection
$$(\nabla_a\nabla_b-\nabla_b\nabla_a)\sigma^{cd}=
W_{ab}{}^c{}_e\sigma^{de}+W_{ab}{}^d{}_e\sigma^{ce}
+2\delta_{[a}{}^c\Rho_{b]e}\sigma^{de}+2\delta_{[a}{}^d\Rho_{b]e}\sigma^{ce}$$
so it follows that
\begin{equation}\label{nearly_there}
W_{ab}{}^c{}_e\sigma^{de}+W_{ab}{}^d{}_e\sigma^{ce}
+2\delta_{[a}{}^c\Psi_{b]}{}^d+2\delta_{[a}{}^d\Psi_{b]}{}^c=0\end{equation}
where
$$\Psi_b{}^d=\nabla_b\mu^d+\Rho_{be}\sigma^{de}.$$
Tracing (\ref{nearly_there}) over ${}_b{}^c$ yields
$$W_{ab}{}^d{}_e\sigma^{be}
-3\Psi_{a}{}^d+\delta_{a}{}^d\Psi_b{}^b=0.$$
Finally, tracing this conclusion over ${}_a{}^d$ shows that $\Psi_b{}^b=0$ and
substituting for $\Psi_b{}^d$ back into (\ref{nearly_there})
gives~(\ref{curvature_constraint}), as required.
\end{proof}
We remark that it is usual to establish (\ref{curvature_constraint}) by firstly
prolonging the metrisability operator, as is done in~\cite{East_Mat}, to obtain
a projectively invariant connection on an auxiliary vector bundle whose
curvature is then computed and found to include the left hand side
of~(\ref{curvature_constraint}). In fact, it is only necessary partially to
prolong the operator before (\ref{curvature_constraint}) emerges and this is
exactly what is done in this proof. Notice that the vanishing of~$\Psi_b{}^b$,
observed at the end of the proof, implies that we can write
$\nabla_b\mu^d=\delta_b{}^d\rho-\Rho_{be}\sigma^{de}$ for some smooth function
$\rho$ and, indeed, this would be the next stage in the general prolongation
procedure advocated in~\cite{BCEG}. This choice loses invariance, and in
\cite{East_Mat} (and also in~\cite[Example~3.4]{HSSS} in accordance with their 
general theory) it is found to be convenient to add some Weyl
curvature to the right hand side
$$\textstyle\nabla_b\mu^d
=\delta_b{}^d\rho-\Rho_{be}\sigma^{de}+\frac13W_{be}{}^d{}_f\sigma^{ef}$$
eventually to obtain a projectively invariant connection. In this article we
shall not pursue this prolongation procedure any further. It is already
observed in~\cite{BDE} that (\ref{curvature_constraint}) gives rise to
non-trivial obstructions to metrisability. To express these obstructions in $3$
dimensions, let us recall that
$W_{ab}{}^c{}_d=\frac12\epsilon_{eab}V^{ec}{}_d$, which enables us to rewrite
(\ref{curvature_constraint}) as
\begin{equation}\label{curvature_constraint_revisited}
V^{(ab}{}_d\sigma^{c)d}=0
\end{equation}
or, in other words, that 
$$\Xi^{abc}{}_{de}\sigma^{de}=0,\quad\mbox{where}\enskip
\Xi^{abc}{}_{de}=V^{(ab}{}_{(d}\delta_{e)}{}^{c)}.$$
Taking projective weights into account, this defines an invariant homomorphism
$$\textstyle\Xi\colon\bigodot^2(TM)(-2)=\xoo{0}{0}{2}\longrightarrow
\bigodot^3(TM)(-6)=\xoo{-3}{0}{3}$$
and Theorem~\ref{prolongation_obstruction} may be recast as follows.
\begin{theo}\label{prolongation_obstruction_revisited}
If $\sigma^{bc}$ solves the metrisability 
equation~\eqref{metrisability_equation}, then $\sigma^{bc}$ lies in the kernel 
of the endomorphism
\begin{equation}\label{basic_endomorphism}\textstyle
\Gamma(M,\bigodot^2(TM)(-2))\ni\sigma^{de}\mapsto
X_a\Xi^{abc}{}_{de}\sigma^{de}\in\Gamma(M,\bigodot^2(TM)(-2))\end{equation}
for any projectively weighted\/ $1$-form $X_a$ of weight~$4$.
\end{theo}

\noindent{\em Proof of Theorem \ref{theorem_1}.}
This is an almost immediate corollary of 
Theorem~\ref{prolongation_obstruction_revisited}. As already observed in the 
introduction, as $\bigodot^2TM(-2)$ is a vector bundle of rank~$6$, 
the determinant of the endomorphism (\ref{basic_endomorphism}) has the form 
$X_aX_bX_cX_dX_eX_fT^{abcdef}$ for some projectively invariant 
$$\textstyle T^{abcdef}\in\Gamma(M,\bigodot^6TM(-24))
=\Gamma(M,\;\xoo{-18}{0}{6}).$$
{\em A priori\/} this might always vanish but a suitable Weyl tensor 
$W_{ab}{}^c{}_d$ is exhibited in \cite[\S8]{BDE} with non-zero determinant 
(and this is realised by a projective structure in accordance with 
Lemma~\ref{all_weyl_tensors_arise}). It remains only to check that 
(\ref{the_determinant}) gives a formula for $T$ but this is easily 
accomplished with the aid of computer algebra.\hfill$\square$

\medskip Since $V^{ab}{}_a=0$, it follows that $\Xi^{abc}{}_{bc}=0$ whence the
endomorphism $X_a\Xi^{abc}{}_{de}$ is traceless for any $X_a$. If we set
$\Xi_X{}^{bc}{}_{de}\equiv X_a\Xi^{abc}{}_{de}$, the Cayley--Hamilton Theorem
for traceless $6\times 6$ matrices now implies that the vanishing of
$T^{abcdef}$ is equivalent to the vanishing of
$$24\,\T({\Xi_X}^6)-18\,\T({\Xi_X}^4)\T({\Xi_X}^2)
-8\,(\T({\Xi_X}^3))^2+3\,(\T({\Xi_X}^2))^3,$$
which is a little easier to compute. Some further consequences of the vanishing
of $T^{abcdef}$ have been analysed in~\cite{casey,nurowski}. Whilst a detailed
analysis of the aforementioned prolonged system (as, for example, presented
in~\cite[Theorem~3.1]{East_Mat}) will surely lead to more obstructions, we do
not pursue this here but, in the following section, opt for an alternative and
more elementary approach. It leads to a plethora of invariant obstructions
whose relation to the metrisability equation and its prolongation remains
mysterious.

\subsection{An elementary construction of obstructions}
In three dimensions, the curvature of a metric connection is entirely 
captured by the Ricci curvature. Specifically,
\begin{equation}\label{Riemannian_decomposition}
\textstyle R_{ab}{}^c{}_d
=\delta_a{}^cR_{bd}-\delta_b{}^cR_{ad}-g_{ad}R_b{}^c+g_{bd}R_a{}^c
-\frac12R(\delta_a{}^cg_{bd}-\delta_b{}^cg_{ad}).\end{equation}
A metric connection is special and from (\ref{defofRho}) we see firstly that 
$\Rho_{ab}=\frac12R_{ab}$ and hence that 
$$\textstyle W_{ab}{}^c{}_d
=R_{ab}{}^c{}_d-\delta_a{}^c\Rho_{bd}+\delta_b{}^c\Rho_{ad}
=R_{ab}{}^c{}_d-\frac12\delta_a{}^cR_{bd}+\frac12\delta_b{}^cR_{ad}.$$
{From} (\ref{Riemannian_decomposition}) we deduce that for a metric connection 
in three dimensions, 
$$\textstyle W_{ab}{}^c{}_d
=\frac12\delta_a{}^cR_{bd}-\frac12\delta_b{}^cR_{ad}
-g_{ad}R_b{}^c+g_{bd}R_a{}^c
-\frac12R\delta_a{}^cg_{bd}+\frac12R\delta_b{}^cg_{ad}.$$
Therefore from its definition (\ref{V_from_W}), we find that
\begin{equation}\label{V_from_def}
V^{ab}{}_c=\epsilon^{dea}W_{de}{}^b{}_c
=\epsilon^{bea}R_{ec}-2\epsilon_c{}^{ea}R_e{}^b
+\epsilon_c{}^{ba}R\end{equation}
\begin{lemma}\label{keylemma} For a metric connection 
$V^{ab}{}_c=2R^{d(a}\epsilon^{b)}{}_{dc}$.
\end{lemma}  
\begin{proof}
Though it might not appear so, the right hand of (\ref{V_from_def}) is 
symmetric in $ab$ as may be verified by computing
$$\epsilon_{dab}(\epsilon^{bea}R_{ec}-2\epsilon_c{}^{ea}R_e{}^b
+\epsilon_c{}^{ba}R)=2\delta_d{}^eR_{ec}-2(g_{bc}\delta_d{}^e
-g_{dc}\delta_b{}^e)R_e{}^b-2g_{dc}R=0.$$
Symmetrising term-by-term in (\ref{V_from_def}) gives the required formula.
\end{proof}

At first glance, it may seem that Lemma~\ref{keylemma} cannot be useful in
restricting the possible Weyl curvature of a metrisable projective structure
because the metric is already involved in the formula for $V^{ab}{}_c$
especially via the tensor $\epsilon^b{}_{dc}$. It turns out, however, that
there are non-trivial projective covariants that necessarily vanish for
$V^{ab}{}_c$ of this special form no matter what metric and no matter what
symmetric form $R^{da}$ are chosen. The simplest example is
$Q_{ab}{}^c=\epsilon_{pq(a}V^{pr}{}_{b)}V^{qc}{}_r$ from
Theorem~\ref{theorem_3}.

\medskip\noindent{\em Proof of Theorem \ref{theorem_3}.} Firstly, let us
establish that $Q_{ab}{}^c$ does not always vanish. A general method that is
almost instantly effective with a computer is simply to compute all the
coefficients of $Q_{ab}{}^c$ as polynomials in, for example, the $15$ variables
\begin{equation}\label{preferred_variables}
\begin{array}{l}
V^{11}{}_2, V^{11}{}_3, V^{21}{}_1, V^{21}{}_2, V^{21}{}_3,
V^{22}{}_1, V^{22}{}_3, V^{31}{}_1,\\ V^{31}{}_2, V^{31}{}_3, 
V^{32}{}_1, V^{32}{}_2, V^{32}{}_3, V^{33}{}_1, V^{33}{}_2,
\end{array}\end{equation}
acting as co\"ordinates on the space
$$\{V^{ab}{}_c\in{\mathbb{R}}^3\otimes{\mathbb{R}}^3\otimes({\mathbb{R}}^3)^*
\mid V^{ab}{}_c=V^{(ab)}{}_c,\;V^{ab}{}_a=0\},$$
any element of which can arise via (\ref{V_from_W}) from the projective Weyl
curvature of a suitable torsion-free connection in accordance with
Lemma~\ref{all_weyl_tensors_arise}. Also, in this calculation, since there is 
only one totally skew $3$-tensor up to scale, we may as well take
$$\epsilon_{123}=\epsilon_{231}=\epsilon_{312}=1\enskip\mbox{and}\enskip
\epsilon_{213}=\epsilon_{132}=\epsilon_{321}=-1.$$
In this case, there is no real need for a computer to obtain, for example,
$$Q_{33}{}^3=V^{11}{}_3 V^{32}{}_1 - V^{31}{}_1 V^{21}{}_3 
           - V^{31}{}_2 V^{22}{}_3 + V^{21}{}_3 V^{32}{}_2$$
and we are done. In fact, for low order invariants such as~$Q_{ab}{}^c$,
glancing ahead to our more systematic investigation starting with
Proposition~\ref{quadratics}, non-vanishing can also be seen without
calculation as follows. The decomposition (\ref{first_decomposition}) proves
the existence of a non-zero covariant $Q_{ab}{}^c=Q_{(ab)}{}^c$ and the only
remaining issue is to find a formula for it. According to Weyl's first
fundamental theorem of invariant theory~\cite{weyl} we are obliged to contract
two copies of $V^{ab}{}_c$ with (by counting the number of covariant and
contravariant indices) one copy of $\epsilon_{abc}$ and then take linear
combinations. Bearing in mind the symmetries of $V^{ab}{}_c$, up to scale the
only possibility for $Q_{ab}{}^c$ is~(\ref{this_is_Q}). This is especially
clear using {\em wiring diagrams\/} as in~\cite{OT}:
$$\raisebox{26pt}{we know}
\begin{picture}(29,40)(10,-16)
% box
\put(23,10){\line(1,0){10}}
\put(23,10){\line(0,1){5}}
\put(23,15){\line(1,0){10}}
\put(33,10){\line(0,1){5}}
% lines
\put(25,15){\line(0,1){10}}
\put(31,15){\line(0,1){10}}
\put(19,25){\line(1,0){12}}
\put(19,0){\line(0,1){25}}
\put(28,0){\line(0,1){10}}
\end{picture}\raisebox{26pt}{$=0$\quad and}
\begin{picture}(29,40)(10,-16)
% box
\put(23,10){\line(1,0){10}}
\put(23,10){\line(0,1){5}}
\put(23,15){\line(1,0){10}}
\put(33,10){\line(0,1){5}}
% lines
\put(31,15){\line(0,1){10}}
\put(19,10){\line(0,1){5}}
% curves
\put(23,10){\oval(8,12)[b]}
\put(22,15){\oval(6,12)[t]}
\end{picture}
\raisebox{26pt}{$=0$\quad therefore\enskip $Q_{ab}{}^c=$}
\begin{picture}(32,40)(10,-5)
% boxes
\put(27,25){\line(1,0){10}}
\put(27,25){\line(0,1){5}}
\put(27,30){\line(1,0){10}}
\put(37,25){\line(0,1){5}}
\put(23,10){\line(1,0){10}}
\put(23,10){\line(0,1){5}}
\put(23,15){\line(1,0){10}}
\put(33,10){\line(0,1){5}}
% lines
\put(35,30){\line(0,1){10}}
\put(29,30){\line(0,1){5}}
\put(24,15){\line(0,1){20}}
\put(32,15){\line(0,1){10}}
\put(19,35){\line(1,0){10}}
\put(19,0){\line(0,1){35}}
\put(28,0){\line(0,1){10}}
% symmetrise
\qbezier (13,5) (15,7) (17,5)
\qbezier (17,5) (19,3) (21,5)
\qbezier (21,5) (23,7) (25,5)
\qbezier (25,5) (27,3) (29,5)
\qbezier (29,5) (31,7) (33,5)
\end{picture}
\raisebox{26pt}{.}$$
Now, we must show that $Q_{ab}{}^c$ vanishes in the metrisable case. Well, if
$V^{ab}{}_c$ has the form given in Lemma~\ref{keylemma}, then we compute
$$\begin{array}{rcl}\epsilon_{pqa}V^{pr}{}_{b}V^{qc}{}_r
&=&\epsilon_{pqa}(R^{dp}\epsilon^r{}_{db}+R^{dr}\epsilon^p{}_{db})
(R^{eq}\epsilon^c{}_{er}+R^{ec}\epsilon^q{}_{er})\\[4pt]
&=&\epsilon_{pqa}\epsilon^r{}_{db}\epsilon^c{}_{er}R^{dp}R^{eq}
  +\epsilon_{pqa}\epsilon^r{}_{db}\epsilon^q{}_{er}R^{dp}R^{ec}\\
&&\quad{}+\epsilon_{pqa}\epsilon^p{}_{db}\epsilon^c{}_{er}R^{dr}R^{eq}
         +\epsilon_{pqa}\epsilon^p{}_{db}\epsilon^q{}_{er}R^{dr}R^{ec}\\[4pt]
&=&\epsilon_{pqa}(\delta_d{}^cg_{be}-\delta_b{}^cg_{de})R^{dp}R^{eq}
  +\epsilon_{pqa}(\delta_d{}^qg_{be}-\delta_b{}^qg_{de})R^{dp}R^{ec}\\
&&\quad{}+(g_{qd}g_{ab}-g_{ad}g_{qb})\epsilon^c{}_{er}R^{dr}R^{eq}
         +(g_{qd}g_{ab}-g_{ad}g_{qb})\epsilon^q{}_{er}R^{dr}R^{ec}\\[4pt]
&=&\epsilon_{pqa}R^{cp}R_b{}^q
  -\epsilon_{pqa}\delta_b{}^cR^{dp}R_d{}^q
  +\epsilon_{pqa}R^{qp}R_b{}^c
  -\epsilon_{pba}R^{dp}R_d{}^c\\
&&\quad{}+g_{ab}\epsilon^c{}_{er}R_q{}^rR^{eq}
         -\epsilon^c{}_{er}R_a{}^rR^e{}_b
         +g_{ab}\epsilon_{der}R^{dr}R^{ec}
         -\epsilon_{ber}R_a{}^rR^{ec}\\[4pt]
&=&\epsilon_{pqa}R^{cp}R_b{}^q
  -\epsilon_{pba}R^{dp}R_d{}^c
         -\epsilon^c{}_{er}R_a{}^rR^e{}_b
         -\epsilon_{ber}R_a{}^rR^{ec}\\[4pt]
&=&\epsilon_{pqa}R^{cp}R_b{}^q
  +\epsilon_{abp}R^{dp}R_d{}^c
         +\epsilon^c{}_{re}R_a{}^rR_b{}^e
         -\epsilon_{bpq}R_a{}^qR^{cp},
\end{array}$$
which is evidently skew in $ab$, as required.\hfill$\square$

\medskip Alternatively, we may compute in a preferred basis, the projective
invariance ensuring that it does not matter what basis is chosen. In the
Riemannian setting, for example, we may chose an orthonormal basis so that
$g_{ab}$ is represented by the identity matrix and in addition choose
$\epsilon_{abc}$ to be the associated volume form. We may also diagonalise
$R^{ab}$ and, optionally, remove its trace since $g^{ab}$ does not contribute
to $V^{ab}{}_c=2R^{d(a}\epsilon^{b)}{}_{dc}$. We leave the resulting
verification to the reader. It is also straightforward to instruct a computer
to work with these normalisations and this is our preferred method for
analysing more complicated projective invariants. Finally, it is sufficient to
work in the Riemannian setting: 
\begin{prop} Working at a point (so that the following statement is purely
algebraic) suppose a projective covariant constructed from $V^{ab}{}_c$
vanishes for all tensors of the form $V^{ab}{}_c=2R^{d(a}g^{b)e}\epsilon_{edc}$
constructed from a fixed positive definite symmetric form $g^{ab}$, an
associated volume form $\epsilon_{abc}$, and an arbitrary trace-free symmetric
form $R^{ab}$. Then the same covariant also vanishes for any
non-degenerate~$g^{ab}$.
\end{prop}
\begin{proof} Clear by complexification.\end{proof}
In particular, when instructing a computer, it is sufficient to assume that
$g^{ab}$ and $R^{ab}$ are simultaneously diagonalised: even though this might
not be possible to arrange in the Lorentzian setting for example, as a
statement of pure algebra it is densely true and this is good enough. In any
case, we shall henceforth suppose that all metrics we encounter are positive
definite.

Lemma~\ref{keylemma} has a representation-theoretic interpretation as follows.
Recall that on a $3$-dimensional projective manifold, the tensor $V^{ab}{}_c$
is a section of \,\xoo{-4}{1}{2}. This is an irreducible bundle but in the
presence of a metric it decomposes according to the branching of the
corresponding representation under
$${\mathrm{GL}}_+(3,{\mathbb{R}})\supset{\mathrm{SO}}(3).$$
Specifically,
$$\xoo{-4}{1}{2}=\Aone{6}\oplus\Aone{4}\oplus\Aone{2}$$
where $\Aone{2k}=\bigodot_\circ^kTM$ and $\circ$ denotes the trace-free tensors
(and \Aone{1}, \Aone{3}, \ldots denote spin bundles that need not concern us
here). By Schur's Lemma, the homomorphisms $\Aone{4}\to\xoo{-4}{1}{2}$ are
unique up to scale so the embedding $\Aone{4}\hookrightarrow\xoo{-4}{1}{2}$ may
as well be realised concretely by Lemma~\ref{keylemma}. The other embeddings 
may as well be realised as
$$\textstyle
\Gamma(M,\bigodot_\circ^3TM)\ni P^{abc}\mapsto P^{ab}{}_c\quad\mbox{and}\quad
\Gamma(M,TM)\ni L^a\mapsto L^{(a}\delta^{b)}{}_c-2g^{ab}L_c.$$

In order to construct potential obstructions by this approach, it is necessary
firstly to construct projective covariants from the tensor $V^{ab}{}_c$. It is
straightforward to compute the locations and dimensions of such covariants.
Quadratic covariants, for example, are limited by the following result.

\begin{prop}\label{quadratics}
Up to scale, there are exactly $5$ distinct quadratic covariants
that may be constructed from~$V^{ab}{}_c$, only one of which vanishes in the
metrisable case.
\end{prop}
\begin{proof} The usual theory of highest weights~\cite{FandH} allows us to
decompose $\bigodot^2(\,\xoo{-4}{1}{2})$ into its irreducible subbundles and 
this is easily implemented with a computer. For example, the 
program LiE~\cite{LiE} decomposes the symmetric tensor power of any 
irreducible representation of any simple algebra. In our case

\verb!sym_tensor(2,[1,2],A2)!    

\noindent returns

\verb!1X[0,2] +1X[1,3] +1X[2,1] +1X[2,4] +1X[4,0]!,

\noindent which implies that 
\begin{equation}\label{first_decomposition}
\textstyle\bigodot^2(\,\xoo{-4}{1}{2})
=\xoo{-6}{0}{2}\oplus\xoo{-7}{1}{3}\oplus\xoo{-7}{2}{1}
\oplus\xoo{-8}{2}{4}\oplus\xoo{-8}{4}{0},\end{equation}
the numbers over the crossed nodes being controlled by the
degree~(\ref{def_of_degree}). This decomposition is exactly what we need to
determine the location and multiplicity of quadratic covariants constructed
from~$V^{ab}{}_c$. The abstract theory does not give formul{\ae} but, according
to Weyl's first fundamental theorem of invariant theory~\cite{weyl}, any such
covariant can be expressed as a linear combination of contractions of
$V^{ab}{}_c$ itself with an appropriate multiplicity (in this case~$2$),
together with the tautological form $\epsilon_{abc}$ or its
inverse~$\epsilon^{abc}$. In our case, we find
\begin{align}
A^{ab}=V^{ap}{}_qV^{bq}{}_p=V^{p(a}{}_qV^{b)q}{}_p
&\in\Gamma(\,\xoo{-6}{0}{2}),\nonumber\\
N^{abc}{}_d=
5V^{(ab}{}_pV^{c)p}{}_d-2A^{(ab}\delta_d{}^{c)}
&\in\Gamma(\,\xoo{-7}{1}{3}),\nonumber\\
Q_{ab}{}^c=\epsilon_{pq(a}V^{pr}{}_{b)}V^{qc}{}_r
&\in\Gamma(\,\xoo{-6}{2}{1}),\nonumber\\
Y^{abcd}{}_{ef}=105V^{(ab}{}_{(e}V^{cd)}{}_{f)}
-12N^{(abc}{}_{(e}\delta^{d)}{}_{f)}
-14A^{(ab}\delta_{(e}{}^c\delta_{f)}{}^{d)}
&\in\Gamma(\,\xoo{-8}{2}{4}),\nonumber\\
Z_{abcd}=\epsilon_{pr(a}V^{pq}{}_bV^{rs}{}_c\epsilon_{d)qs}
&\in\Gamma(\,\xoo{-8}{4}{0}),\nonumber
\end{align}
these formul{\ae} being obtained by trial and error subject to the 
requirements only that the result be a non-vanishing tensor enjoying the 
specified symmetries (for then the formula is guaranteed by Schur's Lemma). It 
is then a matter of computation (best carried out with a computer) to check 
that only the invariant $Q_{ab}{}^c$ vanishes when 
$V^{ab}{}_c=2R^{d(a}\epsilon^{b)}{}_{dc}$.
\end{proof}

Unfortunately, we only know how to prove Proposition~\ref{quadratics} by 
direct calculation. One might hope to prove it by looking at branching under
${\mathrm{SO}}(3)\subset{\mathrm{SL}}_+(3,{\mathbb{R}})$ more carefully. LiE 
easily provides the branching. For example,

\verb!branch([2,4],A1,[[2],[2]],A2)!    

\noindent returns

\verb!1X[0] +2X[4] +1X[6] +2X[8] +1X[10] +1X[12]!,

\noindent which implies that 
$$\xoo{-8}{2}{4}=
\Aone{0}\oplus\Aone{4}\oplus\Aone{4}\oplus\Aone{6}\oplus\Aone{8}
\oplus\Aone{8}\oplus\Aone{10}\oplus\Aone{12}.$$
Similarly,
$$\begin{array}{ll}\xoo{-6}{0}{2}=\Aone{0}\oplus\Aone{4}\qquad&
\xoo{-7}{1}{3}=\Aone{2}\oplus\Aone{4}\oplus\Aone{6}\oplus\Aone{8}\\[4pt]
\xoo{-6}{2}{1}=\Aone{2}\oplus\Aone{4}\oplus\Aone{6}\qquad&
\xoo{-4}{4}{0}=\Aone{0}\oplus\Aone{4}\oplus\Aone{8}
\end{array}$$
whereas invariants arising from $V^{ab}{}_c=2R^{d(a}\epsilon^{b)}{}_{dc}$ must 
lie in  
$$\textstyle\bigodot^2(\Aone{4})=\Aone{0}\oplus\Aone{4}\oplus\Aone{8}.$$
At this point, however, a comparison leads nowhere. 

Whilst we have no theoretical justification for why we might expect
obstructions created in this way, the method of proof given above allows a
systematic though computationally intensive method of finding many more as the
following proofs show.

\medskip\noindent{\em Proof of Theorem \ref{theorem_4}.} 
In fact, noticing that
$$S=\epsilon_{abc}V^{ap}{}_qV^{bq}{}_rV^{cr}{}_p
=V^{ap}{}_q\epsilon_{abc}V^{cr}{}_pV^{bq}{}_r
=-V^{ab}{}_cQ_{ab}{}^c,$$
its vanishing in the metrisable case is immediate corollary of
Theorem~\ref{theorem_3}. A systematic approach to finding~$S$, however, is to
consider the decomposition
\begin{equation}\label{cubic_decomposition}
\begin{array}{rcl}\bigodot^3(\,\xoo{-4}{1}{2})
&=&\xoo{-8}{0}{0}
\oplus\xoo{-9}{0}{3}\oplus\xoo{-9}{0}{3}
\oplus\xoo{-10}{0}{6}
\oplus\xoo{-9}{1}{1}\\[4pt]
&&\enskip{}
\oplus\xoo{-10}{1}{4}\oplus\xoo{-10}{1}{4}
\oplus\xoo{-10}{2}{2}\oplus\xoo{-10}{2}{2}
\oplus\xoo{-11}{2}{5}\\[4pt]
&&\enskip\quad{}
\oplus\xoo{-10}{3}{0}\oplus\xoo{-10}{3}{0}
\oplus\xoo{-11}{3}{3}\oplus\xoo{-11}{3}{3}
\oplus\xoo{-12}{3}{6}\\[4pt]
&&\enskip\qquad{}
\oplus\xoo{-11}{4}{1}\oplus\xoo{-12}{5}{2},
\end{array}\end{equation}
immediately obtained from LiE, write each of them as a linear combination of
contractions, and then test each of these potential obstructions by substituting
$V^{ab}{}_c=2R^{d(a}\epsilon^{b)}{}_{dc}$. This quickly leads to $S$ as stated
in Theorem~\ref{theorem_4}. Finally, the veracity of our claimed formula for 
$S$ may be instantly tested (with a computer) by simply calculating the result 
in our preferred variables (\ref{preferred_variables}), obtaining
$S=6(V^{21}{}_1)^2V^{31}{}_2+3V^{31}{}_1V^{21}{}_1V^{31}{}_3+\cdots$, and 
observing that it is non-zero.
\hfill$\square$

\medskip This reasoning leads to many other cubic covariants. One immediate 
difference to the decomposition~(\ref{first_decomposition}), however, is that 
some subbundles occur with multiplicity. That $\xoo{-9}{0}{3}$ occurs with 
multiplicity~$2$, for example, leads to the covariants
$$\textstyle B^{abc}=\bigodot(V^{ap}{}_qV^{bq}{}_rV^{cr}{}_p)\quad\mbox{and}
\quad C^{abc}=\bigodot(V^{ab}{}_pV^{pq}{}_rV^{cr}{}_q)$$
as stated in Theorem~\ref{theorem_1}. Regarding the first conclusion of
Theorem~\ref{theorem_2}, it remains to check, by direct calculation, that
$B^{abc}$ and $C^{abc}$ are linearly independent in general but that
$C^{abc}=2B^{abc}$ when $V^{ab}{}_c$ is of the form
$2R^{d(a}\epsilon^{b)}{}_{dc}$. In fact, the obstruction $C^{abc}-2B^{abc}$
also arises from $Q_{ab}{}^c$. Specifically, 
$$\begin{array}{rcl}2\epsilon^{pqa}Q_{pr}{}^bV^{cr}{}_q
&=&2\epsilon^{pqa}\epsilon_{de(p}V^{ds}{}_{r)}V^{eb}{}_sV^{cr}{}_q\\[4pt]
&=&\epsilon^{pqa}\epsilon_{dep}V^{ds}{}_{r}V^{eb}{}_sV^{cr}{}_q
+\epsilon^{pqa}\epsilon_{der}V^{ds}{}_{p}V^{eb}{}_sV^{cr}{}_q\\[4pt]
&=&(\delta_d{}^q\delta_e{}^a-\delta_d{}^a\delta_e{}^q)
V^{ds}{}_{r}V^{eb}{}_sV^{cr}{}_q
+\epsilon^{pqa}\epsilon_{der}V^{ds}{}_{p}V^{eb}{}_sV^{cr}{}_q\\[4pt]
&=&V^{qs}{}_{r}V^{ab}{}_sV^{cr}{}_q-V^{as}{}_{r}V^{qb}{}_sV^{cr}{}_q
+(\delta_d{}^p\delta_e{}^q\delta_r{}^a+\cdots)
V^{ds}{}_{p}V^{eb}{}_sV^{cr}{}_q\\[4pt]
&=&2V^{ab}{}_pV^{pq}{}_rV^{cr}{}_q-V^{ac}{}_pV^{pq}{}_rV^{br}{}_q
-2V^{ap}{}_{q}V^{cq}{}_rV^{br}{}_p
\end{array}$$
and so
$$2\epsilon^{pq(a}Q_{pr}{}^bV^{c)r}{}_q=C^{abc}-2B^{abc}.$$

\medskip\noindent{\em Proof of Theorem \ref{theorem_2}.} We have just shown
that $C^{abc}-2B^{abc}=0$ in the metric case, either by direct computation or
as a consequence of the vanishing of the quadratic covariant $Q_{ab}{}^c$ from
Theorem~\ref{theorem_3}. Generally, it is non-zero. The remaining claims in
Theorem~\ref{theorem_2} may be straightforwardly checked by direct computation 
(with a computer).\hfill$\square$

\medskip It is unclear whether all obstructions in Theorem~\ref{theorem_2} may
be written in terms of $Q_{ab}{}^c$. More generally, it is straightforward to
generate many more invariant obstructions all of which may yet arise from the
basic obstruction $Q_{ab}{}^c$. We leave this question for a future
investigation and content ourselves with the following complete determination
of the sextic obstructions taking values in ${}\,\xoo{-18}{0}{6}$ (as does 
$T^{abcddef}$ from Theorem~\ref{theorem_1}).

\begin{theo} There is an $11$-dimensional space of covariants of $V^{ab}{}_c$
of degree $6$ taking values in\/ ${}\,\xoo{-18}{0}{0}$, an $8$-dimensional 
subspace of which vanishes in the metrisable case.
\end{theo}
\begin{proof} The first statement is that ${}\,\xoo{-18}{0}{6}$ occurs with
multiplicity $11$ in $\bigodot^6(\,\xoo{-4}{1}{2})$ and, since
$\deg(\;\xoo{-18}{0}{6})=12=6\times\deg(\,\xoo{-4}{1}{2})$
from~(\ref{def_of_degree}), it suffices to check that the multiplicity of
$\oo{0}{6}$ in $\bigodot^6(\oo{1}{2})$ is~$11$. The LiE~\cite{LiE} command
$$\verb!sym_tensor(6,[1,2],A2)|[0,6]!$$
confirms this (in less than one hundredth of a second). Finding a basis for
this space is then a matter of trial and error. For this purpose, rather than
using indices, it is easier to write covariants using {\em wiring diagrams\/}
as was done in the $19^{\mathrm{th}}$ century~\cite{clifford} (see
also~\cite{OT}). Thus, the covariants $A^{ab}$, $B^{abc}$, and $C^{abc}$ from
Theorem~\ref{theorem_1} are written as
$$\raisebox{18pt}{$A^{ab}=$}
%%%%%%%%%%%%%
% This is A %
%%%%%%%%%%%%%
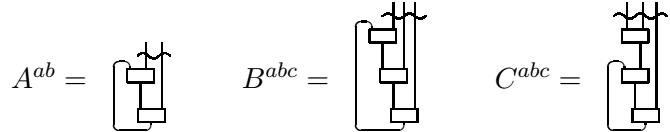
\begin{picture}(36,50)(4,5)
% boxes
\put(19,25){\line(1,0){10}}
\put(19,25){\line(0,1){5}}
\put(19,30){\line(1,0){10}}
\put(29,25){\line(0,1){5}}
\put(23,10){\line(1,0){10}}
\put(23,10){\line(0,1){5}}
\put(23,15){\line(1,0){10}}
\put(33,10){\line(0,1){5}}
% lines
\put(26,30){\line(0,1){10}}
\put(24,15){\line(0,1){10}}
\put(32,15){\line(0,1){25}}
% curves
\put(17,30){\oval(6,6)[t]}
\put(21,10){\oval(14,6)[b]}
\put(17,20){\oval(6,26)[l]}
% symmetrise
\qbezier (23,35) (25,37) (27,35)
\qbezier (27,35) (29,33) (31,35)
\qbezier (31,35) (33,37) (35,35)
\end{picture}
\qquad\raisebox{18pt}{$B^{abc}=$}
%%%%%%%%%%%%%
% This is B %
%%%%%%%%%%%%%
\begin{picture}(40,50)(0,5)
% boxes
\put(15,40){\line(1,0){10}}
\put(15,40){\line(0,1){5}}
\put(15,45){\line(1,0){10}}
\put(25,40){\line(0,1){5}}
\put(19,25){\line(1,0){10}}
\put(19,25){\line(0,1){5}}
\put(19,30){\line(1,0){10}}
\put(29,25){\line(0,1){5}}
\put(23,10){\line(1,0){10}}
\put(23,10){\line(0,1){5}}
\put(23,15){\line(1,0){10}}
\put(33,10){\line(0,1){5}}
% lines
\put(20,30){\line(0,1){10}}
\put(28,30){\line(0,1){25}}
\put(24,15){\line(0,1){10}}
\put(32,15){\line(0,1){40}}
\put(24,45){\line(0,1){10}}
% curves
\put(13,45){\oval(6,6)[t]}
\put(19,10){\oval(18,6)[b]}
\put(13,27.5){\oval(6,41)[l]}
% symmetrise
\qbezier (20,50) (22,52) (24,50)
\qbezier (24,50) (26,48) (28,50)
\qbezier (28,50) (30,52) (32,50)
\qbezier (32,50) (34,48) (36,50)
\end{picture}\qquad
\raisebox{18pt}{$C^{abc}=$}
%%%%%%%%%%%%%
% This is C %
%%%%%%%%%%%%%
\begin{picture}(36,50)(4,5)
% boxes
\put(19,40){\line(1,0){10}}
\put(19,40){\line(0,1){5}}
\put(19,45){\line(1,0){10}}
\put(29,40){\line(0,1){5}}
\put(19,25){\line(1,0){10}}
\put(19,25){\line(0,1){5}}
\put(19,30){\line(1,0){10}}
\put(29,25){\line(0,1){5}}
\put(23,10){\line(1,0){10}}
\put(23,10){\line(0,1){5}}
\put(23,15){\line(1,0){10}}
\put(33,10){\line(0,1){5}}
% lines
\put(26,30){\line(0,1){10}}
\put(24,15){\line(0,1){10}}
\put(32,15){\line(0,1){40}}
\put(27,45){\line(0,1){10}}
\put(21,45){\line(0,1){10}}
% curves
\put(17,30){\oval(6,6)[t]}
\put(21,10){\oval(14,6)[b]}
\put(17,20){\oval(6,26)[l]}
% symmetrise
\qbezier (16,50) (18,52) (20,50)
\qbezier (20,50) (22,48) (24,50)
\qbezier (24,50) (26,52) (28,50)
\qbezier (28,50) (30,48) (32,50)
\qbezier (32,50) (34,52) (36,50)
\end{picture}$$
Recall from~(\ref{cubic_decomposition}) that $\xoo{-9}{0}{3}$ occurs with multiplicity $2$ in 
$\bigodot^3(\,\xoo{-4}{1}{2})$ that $B^{abc}$ and $C^{abc}$ span the 
covariants of this type. At quartic level, 

\verb!sym_tensor(4,[1,2],A2)|[0,4]!

\noindent returns $4$ and we already have $A^{(ab}A^{cd)}$ so we are looking
for $3$ more linearly independent covariants. The following suffice.
%%%%%%%%%%%%%%%%%%%%%%%%%%%%%%%%%%%%%%%%%%%%%%%%%%%%
% NB: sym_tensor(4,[1,2],A2)|[0,4] gives 4         %
%     sym_tensor(4,[1,0,2],A3)|[0,0,4] gives 5     %
%     sym_tensor(4,[1,0,0,2],A4)|[0,0,0,4] gives 5 %
% and now it's 5 all the way down                  %
%%%%%%%%%%%%%%%%%%%%%%%%%%%%%%%%%%%%%%%%%%%%%%%%%%%%
$$\raisebox{18pt}{$D^{abcd}=$}
%%%%%%%%%%%%%
% This is D %
%%%%%%%%%%%%%
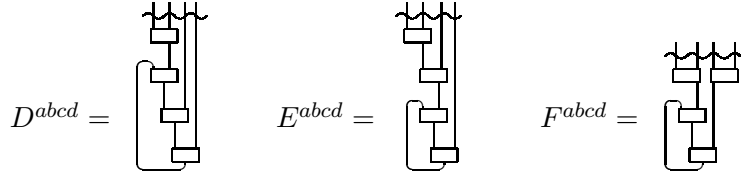
\begin{picture}(40,65)(0,5)
% boxes
\put(15,55){\line(1,0){10}}
\put(15,55){\line(0,1){5}}
\put(15,60){\line(1,0){10}}
\put(25,55){\line(0,1){5}}
\put(15,40){\line(1,0){10}}
\put(15,40){\line(0,1){5}}
\put(15,45){\line(1,0){10}}
\put(25,40){\line(0,1){5}}
\put(19,25){\line(1,0){10}}
\put(19,25){\line(0,1){5}}
\put(19,30){\line(1,0){10}}
\put(29,25){\line(0,1){5}}
\put(23,10){\line(1,0){10}}
\put(23,10){\line(0,1){5}}
\put(23,15){\line(1,0){10}}
\put(33,10){\line(0,1){5}}
% lines
\put(20,30){\line(0,1){10}}
\put(28,30){\line(0,1){40}}
\put(24,15){\line(0,1){10}}
\put(32,15){\line(0,1){55}}
\put(22,45){\line(0,1){10}}
\put(16,60){\line(0,1){10}}
\put(22,60){\line(0,1){10}}
% curves
\put(13,45){\oval(6,6)[t]}
\put(19,10){\oval(18,6)[b]}
\put(13,27.5){\oval(6,41)[l]}
% symmetrise
\qbezier (12,65) (14,67) (16,65)
\qbezier (16,65) (18,63) (20,65)
\qbezier (20,65) (22,67) (24,65)
\qbezier (24,65) (26,63) (28,65)
\qbezier (28,65) (30,67) (32,65)
\qbezier (32,65) (34,63) (36,65)
\end{picture}\qquad
\raisebox{18pt}{$E^{abcd}=$}
%%%%%%%%%%%%%
% This is E %
%%%%%%%%%%%%%
\begin{picture}(40,65)(2,5)
% boxes
\put(13,55){\line(1,0){10}}
\put(13,55){\line(0,1){5}}
\put(13,60){\line(1,0){10}}
\put(23,55){\line(0,1){5}}
\put(19,40){\line(1,0){10}}
\put(19,40){\line(0,1){5}}
\put(19,45){\line(1,0){10}}
\put(29,40){\line(0,1){5}}
\put(19,25){\line(1,0){10}}
\put(19,25){\line(0,1){5}}
\put(19,30){\line(1,0){10}}
\put(29,25){\line(0,1){5}}
\put(23,10){\line(1,0){10}}
\put(23,10){\line(0,1){5}}
\put(23,15){\line(1,0){10}}
\put(33,10){\line(0,1){5}}
% lines
\put(26,30){\line(0,1){10}}
\put(24,15){\line(0,1){10}}
\put(32,15){\line(0,1){55}}
\put(27,45){\line(0,1){25}}
\put(20,45){\line(0,1){10}}
\put(14,60){\line(0,1){10}}
\put(20,60){\line(0,1){10}}
% curves
\put(17,30){\oval(6,6)[t]}
\put(21,10){\oval(14,6)[b]}
\put(17,20){\oval(6,26)[l]}
% symmetrise
\qbezier (10,65) (12,67) (14,65)
\qbezier (14,65) (16,63) (18,65)
\qbezier (18,65) (20,67) (22,65)
\qbezier (22,65) (24,63) (26,65)
\qbezier (26,65) (28,67) (30,65)
\qbezier (30,65) (32,63) (34,65)
\end{picture}\qquad
\raisebox{18pt}{$F^{abcd}=$}
%%%%%%%%%%%%%
% This is F %
%%%%%%%%%%%%%
\begin{picture}(36,50)(4,5)
% boxes
\put(17,40){\line(1,0){10}}
\put(17,40){\line(0,1){5}}
\put(17,45){\line(1,0){10}}
\put(27,40){\line(0,1){5}}
\put(31,40){\line(1,0){10}}
\put(31,40){\line(0,1){5}}
\put(31,45){\line(1,0){10}}
\put(41,40){\line(0,1){5}}
\put(19,25){\line(1,0){10}}
\put(19,25){\line(0,1){5}}
\put(19,30){\line(1,0){10}}
\put(29,25){\line(0,1){5}}
\put(23,10){\line(1,0){10}}
\put(23,10){\line(0,1){5}}
\put(23,15){\line(1,0){10}}
\put(33,10){\line(0,1){5}}
% lines
\put(18,45){\line(0,1){10}}
\put(26,45){\line(0,1){10}}
\put(32,45){\line(0,1){10}}
\put(40,45){\line(0,1){10}}
\put(26,30){\line(0,1){10}}
\put(24,15){\line(0,1){10}}
\put(32,15){\line(0,1){25}}
% curves
\put(17,30){\oval(6,6)[t]}
\put(21,10){\oval(14,6)[b]}
\put(17,20){\oval(6,26)[l]}
% symmetrise
\qbezier (14,50) (16,52) (18,50)
\qbezier (18,50) (20,48) (22,50)
\qbezier (22,50) (24,52) (26,50)
\qbezier (26,50) (28,48) (30,50)
\qbezier (30,50) (32,52) (34,50)
\qbezier (34,50) (36,48) (38,50)
\qbezier (38,50) (40,52) (42,50)
\end{picture}$$
Note that one might na\"{\i}vely also expect to encounter the covariant
$$
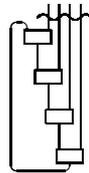
\begin{picture}(40,65)(0,5)
% boxes
\put(11,55){\line(1,0){10}}
\put(11,55){\line(0,1){5}}
\put(11,60){\line(1,0){10}}
\put(21,55){\line(0,1){5}}
\put(15,40){\line(1,0){10}}
\put(15,40){\line(0,1){5}}
\put(15,45){\line(1,0){10}}
\put(25,40){\line(0,1){5}}
\put(19,25){\line(1,0){10}}
\put(19,25){\line(0,1){5}}
\put(19,30){\line(1,0){10}}
\put(29,25){\line(0,1){5}}
\put(23,10){\line(1,0){10}}
\put(23,10){\line(0,1){5}}
\put(23,15){\line(1,0){10}}
\put(33,10){\line(0,1){5}}
% lines
\put(16,45){\line(0,1){10}}
\put(20,30){\line(0,1){10}}
\put(28,30){\line(0,1){40}}
\put(24,15){\line(0,1){10}}
\put(32,15){\line(0,1){55}}
\put(24,45){\line(0,1){25}}
\put(20,60){\line(0,1){10}}
% curves
\put(9,60){\oval(6,6)[t]}
\put(17,10){\oval(22,6)[b]}
\put(9,35){\oval(6,56)[l]}
% symmetrise
\qbezier (16,65) (18,67) (20,65)
\qbezier (20,65) (22,63) (24,65)
\qbezier (24,65) (26,67) (28,65)
\qbezier (28,65) (30,63) (32,65)
\qbezier (32,65) (34,67) (36,65)
\end{picture}$$
and in higher dimensions this would, indeed, be an independent covariant. In
$3$ dimensions, however, it turns out that this is $\frac12A^{(ab}A^{cd)}$. In
principle, such relations are covered by Weyl's second fundamental theorem of
invariant theory~\cite{weyl}, which says that they all arise by `skewing over
too many indices,' in this case~$4$ (the Cayley-Hamilton Theorem being a
familiar example of such dimension-dependent relations). In practise, however, 
it is best to write down all potential covariants and allow a computer to sort 
out the relations.

For completeness, we now move on to quintic invariants. Since 

\verb!sym_tensor(5,[1,2],A2)|[0,5]!

\noindent returns $5$ we need $3$ more invariants to complement 
$A^{(ab}B^{cde)}$ and $A^{(ab}C^{cde)}$. It turns out that
$$\raisebox{18pt}{$G^{abcde}=$}
%%%%%%%%%%%%%
% This is G %
%%%%%%%%%%%%%
\begin{picture}(40,80)(0,5)
% boxes
\put(8,70){\line(1,0){10}}
\put(8,70){\line(0,1){5}}
\put(8,75){\line(1,0){10}}
\put(18,70){\line(0,1){5}}
\put(15,55){\line(1,0){10}}
\put(15,55){\line(0,1){5}}
\put(15,60){\line(1,0){10}}
\put(25,55){\line(0,1){5}}
\put(15,40){\line(1,0){10}}
\put(15,40){\line(0,1){5}}
\put(15,45){\line(1,0){10}}
\put(25,40){\line(0,1){5}}
\put(19,25){\line(1,0){10}}
\put(19,25){\line(0,1){5}}
\put(19,30){\line(1,0){10}}
\put(29,25){\line(0,1){5}}
\put(23,10){\line(1,0){10}}
\put(23,10){\line(0,1){5}}
\put(23,15){\line(1,0){10}}
\put(33,10){\line(0,1){5}}
% lines
\put(20,30){\line(0,1){10}}
\put(28,30){\line(0,1){55}}
\put(24,15){\line(0,1){10}}
\put(32,15){\line(0,1){70}}
\put(22,45){\line(0,1){10}}
\put(16,60){\line(0,1){10}}
\put(22,60){\line(0,1){25}}
\put(16,75){\line(0,1){10}}
\put(9,75){\line(0,1){10}}
% curves
\put(13,45){\oval(6,6)[t]}
\put(19,10){\oval(18,6)[b]}
\put(13,27.5){\oval(6,41)[l]}
% symmetrise
\qbezier (5,80) (7,82) (9,80)
\qbezier (9,80) (11,78) (13,80)
\qbezier (13,80) (15,82) (17,80)
\qbezier (17,80) (19,78) (21,80)
\qbezier (21,80) (23,82) (25,80)
\qbezier (25,80) (27,78) (29,80)
\qbezier (29,80) (31,82) (33,80)
\qbezier (33,80) (35,78) (37,80)
\end{picture}\qquad
\raisebox{18pt}{$H^{abcde}=$}
%%%%%%%%%%%%%
% This is H %
%%%%%%%%%%%%%
\begin{picture}(40,65)(2,5)
% boxes
\put(25,55){\line(1,0){10}}
\put(25,55){\line(0,1){5}}
\put(25,60){\line(1,0){10}}
\put(35,55){\line(0,1){5}}
\put(13,55){\line(1,0){10}}
\put(13,55){\line(0,1){5}}
\put(13,60){\line(1,0){10}}
\put(23,55){\line(0,1){5}}
\put(19,40){\line(1,0){10}}
\put(19,40){\line(0,1){5}}
\put(19,45){\line(1,0){10}}
\put(29,40){\line(0,1){5}}
\put(19,25){\line(1,0){10}}
\put(19,25){\line(0,1){5}}
\put(19,30){\line(1,0){10}}
\put(29,25){\line(0,1){5}}
\put(23,10){\line(1,0){10}}
\put(23,10){\line(0,1){5}}
\put(23,15){\line(1,0){10}}
\put(33,10){\line(0,1){5}}
% lines
\put(26,30){\line(0,1){10}}
\put(24,15){\line(0,1){10}}
\put(28,60){\line(0,1){10}}
\put(34,60){\line(0,1){10}}
\put(28,45){\line(0,1){10}}
\put(20,45){\line(0,1){10}}
\put(14,60){\line(0,1){10}}
\put(20,60){\line(0,1){10}}
\put(40,21){\line(0,1){49}}
% curves
\put(17,30){\oval(6,6)[t]}
\put(21,10){\oval(14,6)[b]}
\put(17,20){\oval(6,26)[l]}
\put(36,15){\oval(8,6)[tl]}
\put(36,21){\oval(8,6)[br]}
% symmetrise
\qbezier (10,65) (12,67) (14,65)
\qbezier (14,65) (16,63) (18,65)
\qbezier (18,65) (20,67) (22,65)
\qbezier (22,65) (24,63) (26,65)
\qbezier (26,65) (28,67) (30,65)
\qbezier (30,65) (32,63) (34,65)
\qbezier (34,65) (36,67) (38,65)
\qbezier (38,65) (40,63) (42,65)
\end{picture}\qquad
\raisebox{18pt}{$I^{abcde}=$}
%%%%%%%%%%%%%
% This is I %
%%%%%%%%%%%%%
\begin{picture}(36,50)(4,5)
% boxes
\put(17,40){\line(1,0){10}}
\put(17,40){\line(0,1){5}}
\put(17,45){\line(1,0){10}}
\put(27,40){\line(0,1){5}}
\put(31,40){\line(1,0){10}}
\put(31,40){\line(0,1){5}}
\put(31,45){\line(1,0){10}}
\put(41,40){\line(0,1){5}}
\put(19,25){\line(1,0){10}}
\put(19,25){\line(0,1){5}}
\put(19,30){\line(1,0){10}}
\put(29,25){\line(0,1){5}}
\put(37,25){\line(1,0){10}}
\put(37,25){\line(0,1){5}}
\put(37,30){\line(1,0){10}}
\put(47,25){\line(0,1){5}}
\put(23,10){\line(1,0){10}}
\put(23,10){\line(0,1){5}}
\put(23,15){\line(1,0){10}}
\put(33,10){\line(0,1){5}}
% lines
\put(18,45){\line(0,1){10}}
\put(26,45){\line(0,1){10}}
\put(32,45){\line(0,1){10}}
\put(40,45){\line(0,1){10}}
\put(26,30){\line(0,1){10}}
\put(24,15){\line(0,1){10}}
\put(38,30){\line(0,1){10}}
\put(46,30){\line(0,1){25}}
\put(40,21){\line(0,1){4}}
% curves
\put(17,30){\oval(6,6)[t]}
\put(21,10){\oval(14,6)[b]}
\put(17,20){\oval(6,26)[l]}
\put(36,15){\oval(8,6)[tl]}
\put(36,21){\oval(8,6)[br]}
% symmetrise
\qbezier (14,50) (16,52) (18,50)
\qbezier (18,50) (20,48) (22,50)
\qbezier (22,50) (24,52) (26,50)
\qbezier (26,50) (28,48) (30,50)
\qbezier (30,50) (32,52) (34,50)
\qbezier (34,50) (36,48) (38,50)
\qbezier (38,50) (40,52) (42,50)
\qbezier (42,50) (44,48) (46,50)
\qbezier (46,50) (48,52) (50,50)
\end{picture}$$
will suffice (amongst the $7$ possible quintic wiring diagrams beyond 
$A^{(ab}B^{cde)}$ and $A^{(ab}C^{cde)}$).

Finally we come to identify the sextic invariants and already we have 
$$\begin{array}{c}
A^{(ab}A^{cd}A^{ef)},\qquad B^{(abc}B^{def)},\qquad C^{(abc}C^{def)},\\[4pt]
B^{(abc}C^{def)},\qquad A^{(ab}D^{cdef)},\qquad A^{(ab}E^{cdef)},\qquad 
A^{(ab}F^{cdef)}.
\end{array}$$
It remains to find $4$ more and it turns out that 
$$\raisebox{18pt}{$J^{abcde}=$}
%%%%%%%%%%%%%
% This is J %
%%%%%%%%%%%%%
\begin{picture}(42,80)(4,5)
% boxes
\put(15,70){\line(1,0){10}}
\put(15,70){\line(0,1){5}}
\put(15,75){\line(1,0){10}}
\put(25,70){\line(0,1){5}}
\put(28,70){\line(1,0){10}}
\put(28,70){\line(0,1){5}}
\put(28,75){\line(1,0){10}}
\put(38,70){\line(0,1){5}}
\put(20,55){\line(1,0){10}}
\put(20,55){\line(0,1){5}}
\put(20,60){\line(1,0){10}}
\put(30,55){\line(0,1){5}}
\put(19,25){\line(1,0){10}}
\put(19,25){\line(0,1){5}}
\put(19,30){\line(1,0){10}}
\put(29,25){\line(0,1){5}}
\put(20,40){\line(1,0){10}}
\put(20,40){\line(0,1){5}}
\put(20,45){\line(1,0){10}}
\put(30,40){\line(0,1){5}}
\put(23,10){\line(1,0){10}}
\put(23,10){\line(0,1){5}}
\put(23,15){\line(1,0){10}}
\put(33,10){\line(0,1){5}}
% lines
\put(23,45){\line(0,1){10}}
\put(26,30){\line(0,1){10}}
\put(24,15){\line(0,1){10}}
\put(46,21){\line(0,1){64}}
\put(41,51){\line(0,1){34}}
\put(17,75){\line(0,1){10}}
\put(23,75){\line(0,1){10}}
\put(29,75){\line(0,1){10}}
\put(35,75){\line(0,1){10}}
\put(21,60){\line(0,1){10}}
\put(29,60){\line(0,1){10}}
% curves
\put(17,30){\oval(6,6)[t]}
\put(21,10){\oval(14,6)[b]}
\put(17,20){\oval(6,26)[l]}
\put(39,15){\oval(14,6)[tl]}
\put(39,21){\oval(14,6)[br]}
\put(35,45){\oval(12,6)[tl]}
\put(35,51){\oval(12,6)[br]}
% symmetrise
\qbezier (13,80) (15,82) (17,80)
\qbezier (17,80) (19,78) (21,80)
\qbezier (21,80) (23,82) (25,80)
\qbezier (25,80) (27,78) (29,80)
\qbezier (29,80) (31,82) (33,80)
\qbezier (33,80) (35,78) (37,80)
\qbezier (37,80) (39,82) (41,80)
\qbezier (41,80) (43,78) (45,80)
\qbezier (45,80) (47,82) (49,80)
\end{picture}\qquad
\raisebox{18pt}{$K^{abcde}=$}
%%%%%%%%%%%%%
% This is K %
%%%%%%%%%%%%%
\begin{picture}(40,65)(4,5)
% boxes
\put(15,55){\line(1,0){10}}
\put(15,55){\line(0,1){5}}
\put(15,60){\line(1,0){10}}
\put(25,55){\line(0,1){5}}
\put(33,55){\line(1,0){10}}
\put(33,55){\line(0,1){5}}
\put(33,60){\line(1,0){10}}
\put(43,55){\line(0,1){5}}
\put(20,40){\line(1,0){10}}
\put(20,40){\line(0,1){5}}
\put(20,45){\line(1,0){10}}
\put(30,40){\line(0,1){5}}
\put(19,25){\line(1,0){10}}
\put(19,25){\line(0,1){5}}
\put(19,30){\line(1,0){10}}
\put(29,25){\line(0,1){5}}
\put(37,40){\line(1,0){10}}
\put(37,40){\line(0,1){5}}
\put(37,45){\line(1,0){10}}
\put(47,40){\line(0,1){5}}
\put(23,10){\line(1,0){10}}
\put(23,10){\line(0,1){5}}
\put(23,15){\line(1,0){10}}
\put(33,10){\line(0,1){5}}
% lines
\put(21,45){\line(0,1){10}}
\put(29,45){\line(0,1){25}}
\put(38,45){\line(0,1){10}}
\put(26,30){\line(0,1){10}}
\put(24,15){\line(0,1){10}}
\put(40,21){\line(0,1){19}}
\put(17,60){\line(0,1){10}}
\put(23,60){\line(0,1){10}}
\put(35,60){\line(0,1){10}}
\put(41,60){\line(0,1){10}}
\put(46,45){\line(0,1){25}}
% curves
\put(17,30){\oval(6,6)[t]}
\put(21,10){\oval(14,6)[b]}
\put(17,20){\oval(6,26)[l]}
\put(36,15){\oval(8,6)[tl]}
\put(36,21){\oval(8,6)[br]}
% symmetrise
\qbezier (13,65) (15,67) (17,65)
\qbezier (17,65) (19,63) (21,65)
\qbezier (21,65) (23,67) (25,65)
\qbezier (25,65) (27,63) (29,65)
\qbezier (29,65) (31,67) (33,65)
\qbezier (33,65) (35,63) (37,65)
\qbezier (37,65) (39,67) (41,65)
\qbezier (41,65) (43,63) (45,65)
\qbezier (45,65) (47,67) (49,65)
\end{picture}\qquad
\raisebox{18pt}{$L^{abcde}=$}
%%%%%%%%%%%%%
% This is L %
%%%%%%%%%%%%%
\begin{picture}(40,80)(4,5)
% boxes
\put(15,70){\line(1,0){10}}
\put(15,70){\line(0,1){5}}
\put(15,75){\line(1,0){10}}
\put(25,70){\line(0,1){5}}
\put(28,70){\line(1,0){10}}
\put(28,70){\line(0,1){5}}
\put(28,75){\line(1,0){10}}
\put(38,70){\line(0,1){5}}
\put(20,55){\line(1,0){10}}
\put(20,55){\line(0,1){5}}
\put(20,60){\line(1,0){10}}
\put(30,55){\line(0,1){5}}
\put(20,25){\line(1,0){10}}
\put(20,25){\line(0,1){5}}
\put(20,30){\line(1,0){10}}
\put(30,25){\line(0,1){5}}
\put(19,40){\line(1,0){10}}
\put(19,40){\line(0,1){5}}
\put(19,45){\line(1,0){10}}
\put(29,40){\line(0,1){5}}
\put(23,10){\line(1,0){10}}
\put(23,10){\line(0,1){5}}
\put(23,15){\line(1,0){10}}
\put(33,10){\line(0,1){5}}
% lines
\put(26,45){\line(0,1){10}}
\put(22,30){\line(0,1){10}}
\put(24,15){\line(0,1){10}}
\put(46,21){\line(0,1){64}}
\put(41,36){\line(0,1){49}}
\put(17,75){\line(0,1){10}}
\put(23,75){\line(0,1){10}}
\put(29,75){\line(0,1){10}}
\put(35,75){\line(0,1){10}}
\put(21,60){\line(0,1){10}}
\put(29,60){\line(0,1){10}}
% curves
\put(17,45){\oval(6,6)[t]}
\put(21,10){\oval(14,6)[b]}
\put(17,27.5){\oval(6,41)[l]}
\put(39,15){\oval(14,6)[tl]}
\put(39,21){\oval(14,6)[br]}
\put(35,30){\oval(12,6)[tl]}
\put(35,36){\oval(12,6)[br]}
% symmetrise
\qbezier (13,80) (15,82) (17,80)
\qbezier (17,80) (19,78) (21,80)
\qbezier (21,80) (23,82) (25,80)
\qbezier (25,80) (27,78) (29,80)
\qbezier (29,80) (31,82) (33,80)
\qbezier (33,80) (35,78) (37,80)
\qbezier (37,80) (39,82) (41,80)
\qbezier (41,80) (43,78) (45,80)
\qbezier (45,80) (47,82) (49,80)
\end{picture}\qquad
\raisebox{18pt}{$M^{abcde}=$}
%%%%%%%%%%%%%
% This is M %
%%%%%%%%%%%%%
\begin{picture}(54,80)(4,35)
% boxes
\put(15,70){\line(1,0){10}}
\put(15,70){\line(0,1){5}}
\put(15,75){\line(1,0){10}}
\put(25,70){\line(0,1){5}}
\put(28,70){\line(1,0){10}}
\put(28,70){\line(0,1){5}}
\put(28,75){\line(1,0){10}}
\put(38,70){\line(0,1){5}}
\put(20,55){\line(1,0){10}}
\put(20,55){\line(0,1){5}}
\put(20,60){\line(1,0){10}}
\put(30,55){\line(0,1){5}}
\put(41,55){\line(1,0){10}}
\put(41,55){\line(0,1){5}}
\put(41,60){\line(1,0){10}}
\put(51,55){\line(0,1){5}}
\put(25,40){\line(1,0){10}}
\put(25,40){\line(0,1){5}}
\put(25,45){\line(1,0){10}}
\put(35,40){\line(0,1){5}}
\put(41,70){\line(1,0){10}}
\put(41,70){\line(0,1){5}}
\put(41,75){\line(1,0){10}}
\put(51,70){\line(0,1){5}}
% lines
\put(26,45){\line(0,1){10}}
\put(45,51){\line(0,1){4}}
\put(17,75){\line(0,1){10}}
\put(23,75){\line(0,1){10}}
\put(29,75){\line(0,1){10}}
\put(35,75){\line(0,1){10}}
\put(43,75){\line(0,1){10}}
\put(49,75){\line(0,1){10}}
\put(21,60){\line(0,1){10}}
\put(29,60){\line(0,1){10}}
\put(43,60){\line(0,1){10}}
% curves
\put(53,60){\oval(6,6)[t]}
\put(44,40){\oval(24,6)[b]}
\put(53,50){\oval(6,26)[r]}
\put(39,45){\oval(12,6)[tl]}
\put(39,51){\oval(12,6)[br]}
% symmetrise
\qbezier (13,80) (15,82) (17,80)
\qbezier (17,80) (19,78) (21,80)
\qbezier (21,80) (23,82) (25,80)
\qbezier (25,80) (27,78) (29,80)
\qbezier (29,80) (31,82) (33,80)
\qbezier (33,80) (35,78) (37,80)
\qbezier (37,80) (39,82) (41,80)
\qbezier (41,80) (43,78) (45,80)
\qbezier (45,80) (47,82) (49,80)
\qbezier (49,80) (51,78) (53,80)
\end{picture}$$
will suffice. It is then a matter of computation (with a computer) to check 
that, in addition to the $5$ obstructions listed in 
Theorem~\ref{theorem_2}, there are two more, namely
$$E^{abcd}-F^{abcd}\quad\mbox{and}\quad J^{abcdef}-M^{abcdef},$$
but that there are no more relations amongst $A,B,C,D,E,F,J,K,L$. Therefore, 
amongst the $11$-dimensional space of covariants of degree $6$, the subspace 
comprising those that vanish in the metrisable case is $8$-dimensional and is
spanned by 
$$\begin{array}{c}
J-M,\quad J-2L,\quad 3J-2C\odot C,\quad J-4K+4A\odot D,\\[4pt]
A\odot E-A\odot F,\quad A\odot F-2A\odot D,\quad B\odot C-2B\odot  B,\quad 
C\odot C-2B\odot C,
\end{array}$$
for example.
\end{proof}

\subsection{Egorov's projective structure}\label{egorov_subsection}
A {\em projective symmetry\/} is a vector field whose local flow maps
unparametrised geodesics to unparametrised geodesics. Such symmetries form a
Lie algebra ${\mathfrak{g}}$ under Lie bracket and, for connected
$3$-dimensional projective structures, we have $\dim{\mathfrak{g}}\leq 15$ with
equality if and only if ${\mathfrak{g}}={\mathfrak{sl}}(4,{\mathbb{R}})$ in
which case the structure is projectively flat, equivalently
$V^{ab}{}_c\equiv0$.

The submaximal dimension for simply-connected $3$-dimensional projective
structures is~$8$. The corresponding projective structure was given by
Egorov~\cite{egorov} and can be represented in local co\"ordinates
$(x^1,x^2,x^3)$ as 
\be
\label{egorov_p}
\nabla_2X^1=\partial_2X^1+x^2X^3,\quad
\nabla_3X^1=\partial_3X^1+x^2X^2,\quad\mbox{else}\quad
\nabla_aX^c=\partial_aX^c,
\ee 
where $\partial_a\equiv\partial/\partial x^a$.
For this connection, the only non-zero components of curvature are 
$$R_{23}{}^1{}_2=-1\qquad R_{32}{}^1{}_2=1$$
so $R_{ab}{}^c{}_d$ is already trace-free. Therefore 
$W_{ab}{}^c{}_d=R_{ab}{}^c{}_d$. Hence
\begin{equation}\label{egorov_curvature}
V^{11}{}_2=-2\quad\mbox{with all other components zero}.\end{equation}
Immediately (\ref{curvature_constraint_revisited}) implies 
$\sigma^{12}=\sigma^{22}=\sigma^{23}=0$ and then we can solve 
(\ref{metrisability_equation}) explicitly:
\begin{equation}\label{explicit_solution}\sigma^{ab}=\left[\begin{array}{ccc}
A-B(x^2)^2+C(x^2)^4&0&B-2C(x^2)^2\\
0&0&0\\
B-2C(x^2)^2&0&4C
\end{array}\right]\end{equation}
for arbitrary constants $A,B,C$. Since all solutions are degenerate, we have
shown: 
\begin{prop} The Egorov projective structure is not metrisable.
\end{prop}
\noindent{\em Another Proof.} As soon as the dimension of the solution space to
(\ref{metrisability_equation}) reaches~$3$, the projective structure cannot be
metrisable unless it is projectively flat (it is easy to check that
(\ref{explicit_solution}) solves (\ref{metrisability_equation}) and that the
projective Weyl curvature is non-vanishing without knowing that
(\ref{explicit_solution}) is the general solution). In general, the 
{\em degree of mobility\/} of a metric is the dimension of the solution space
of (\ref{metrisability_equation}) for the associated projective structure and
in $3$ dimensions it can only be $1$, $2$, or $10$ (as shown in~\cite{KM} in
the Riemannian case and \cite{K} in the Lorentzian case (see
\cite{KMMS,Shandra} in the Riemannian setting and \cite {FM} in the Lorentzian
setting for a detailed analysis concerning possible values of the degree of
mobility in all dimensions (in~\cite{FM}, a detailed analysis is conducted
under the assumption that there are at least two metrics in the projective
class whose corresponding Levi-Civita connections are different but if this is
not the case, then this is a sufficient imposition on the projective Weyl
curvature that it must vanish))). Alternatively, Kruglikov and
Matveev~\cite{BSK_VSM} consider the dimension of the space of local projective
symmetries to conclude that the Egorov structure is not metrisable.
Specifically, in $3$ dimensions they show that the dimension of this space is
bounded by $5$ if there is a Riemannian metric inducing the projective
structure and $6$ if there is a Lorentzian metric inducing the projective
structure (whereas, as noted above, the local projective symmetries are
$8$-dimensional for the Egorov structure). \hfill$\square$

\medskip\noindent{\em Yet Another Proof.} 
We consider the projective Weyl curvature of the
Egorov structure in comparison with Lemma~\ref{keylemma}. To make the
comparison, let us use the metric to lower indices:
\begin{equation}\label{hook_symmetry}V_{abc}=2R^d{}_{(a}\epsilon_{b)dc},
\enskip\mbox{which implies that}\enskip V_{(abc)}=0.
\end{equation}
As a consequence of~(\ref{egorov_curvature}), however, $V^{ab}{}_c$ is simple,
i.e.~there are non-zero $X^a$ and $Y_a$ such that $V^{ab}{}_c=X^aX^bY_c$.
Therefore we have $V_{(abc)}=X_{(a}X_{\vphantom{(}b}Y_{c)}\not=0$, in conflict
with~(\ref{hook_symmetry}).
% If the Egorov structure were metrisable then, according to this lemma, in any
% orthonormal frame we should have
% $$V^{11}{}_2=2R^{d1}\epsilon^{1}{}_{d2}=-2R^{31}=-2R^{13}
% =-2R^{d3}\epsilon^{3}{}_{d2}=-V^{33}{}_2,$$
% which na\"{\i}vely seems to be in conflict with~(\ref{egorov_curvature}). Of
% course, one is allowed an arbitrary change of frame in order to resolve this
% conflict but, by the Gram-Schmidt process, an arbitrary real non-singular
% matrix can be written as~$KAN$, where $K$ is orthogonal, $A$~is diagonal, and
% $N$ is strictly upper-triangular with $1$s down the diagonal. The conflict 
% with (\ref{egorov_curvature}) remains and so we have yet another proof that 
% the Egorov structure is not metrisable with a Riemannian metric. To exclude a
% Lorentzian metric one needs to take care that our frame may contain null
% vectors but always the conflict with (\ref{egorov_curvature}) remains.
\hfill$\square$

\medskip
Thus, we have yet another proof, this one just from the Weyl curvature, that
the Egorov projective structure is not metrisable.
On the other hand, for $V^{ab}{}_c$ as in~(\ref{egorov_curvature}), projecting 
$V\odot V\odot \cdots\odot V$ into any component of 
$$\textstyle\bigodot^d(\,\xoo{-4}{1}{2})=\;\xoo{-4d\;{}}{d}{2d}\oplus\cdots$$
other than the first evidently gives zero. Thus, we see that the Weyl 
curvature is sufficiently special that the structure cannot be metrisable but 
that this situation cannot be detected by the vanishing of any projective 
covariant. 

Finally, we remark that, although the Weyl tensor $V^{ab}{}_c$ for the Egorov
structure does not have the form required by Lemma~\ref{keylemma}, it can still
be written as
$$V^{ab}{}_c=2R^{d(a}g^{b)e}\epsilon_{edc},\quad\mbox{where}\enskip
R^{ab}=\left[\begin{array}{ccc}0&0&1\\ 0&0&0\\ 1&0&0\end{array}\right]
\enskip\mbox{and}\enskip 
g^{ab}=\left[\begin{array}{ccc}1&0&0\\ 0&0&0\\ 0&0&0\end{array}\right],$$
which is of the required form save for $g^{ab}$ being degenerate.

\subsection{Newtonian projective structures}\label{newtonian_subsection}
%%%%%%%%%%%%%%%%%%%%%%%%%%%%%%%%%%%%%%%%%%%%%%%%
%% Initially MikE thought the following would %%
%% be inevitable but instad it is false!      %%
%%%%%%%%%%%%%%%%%%%%%%%%%%%%%%%%%%%%%%%%%%%%%%%%
% In this section, we shall find non-metrisable projective structures whose 
% Weyl curvature, nevertheless, has the form 
% $V^{ab}{}_c=2S^{d(a}\epsilon^{b)}{}_{dc}$ for some symmetric 
% tensor~$S^{ab}$, as in Lemma~\ref{keylemma}. 
% Therefore, any obstruction due to Lemma~\ref{keylemma} must vanish. 
In this section we construct some more non-metrisable projective
structures notwithstanding that, again, all our obstructions vanish. Indeed,
any obstruction that depends continuously on the projective structure will
vanish for these {\em Newtonian\/} projective structures because they are
created as limits of metrisable structures as follows. In local co\"ordinates
$(x^1,x^2,x^3)$, consider the metric
$$\epsilon\big((dx^1)^2+(dx^2)^2\big)+\exp(\epsilon f(x^1,x^2))(dx^3)^2$$
where $\epsilon\not=0$ is constant and $f(x^1,x^2)$ is an arbitrary smooth 
function. The corresponding projective structures are metrisable by definition 
but if we let $\epsilon\to0$, then these metric connections have a perfectly 
good limit, namely
\begin{equation}\label{newton}
\textstyle\nabla_3X^1=\partial_3X^1-\frac12(\partial_1f)X^3,\quad
\nabla_3X^2=\partial_3X^2-\frac12(\partial_2f)X^3,\quad\mbox{else}\quad
\nabla_aX^c=\partial_aX^c,\end{equation} 
(whose geodesic equations are Newton's equations for a particle
in the $(x^1,x^2)$-plane moving under the influence of the potential
$f(x^1,x^2)$ with $x^3=$ `time') whereas 
\begin{prop} Unless projectively flat, the Newtonian projective structures
\eqref{newton} are not metrisable.
\end{prop}
\begin{proof}
Firstly, we compute the curvature of (\ref{newton}) 
$$\begin{array}{llll}
R_{13}{}^1{}_3=-\frac12(\partial_1)^2f
&R_{13}{}^2{}_3=-\frac12\partial_1\partial_2f
&R_{23}{}^1{}_3=-\frac12\partial_1\partial_2f
&R_{23}{}^2{}_3=-\frac12(\partial_2)^2f\\
R_{31}{}^1{}_3=\frac12(\partial_1)^2f
&R_{31}{}^2{}_3=\frac12\partial_1\partial_2f
&R_{32}{}^1{}_3=\frac12\partial_1\partial_2f
&R_{32}{}^2{}_3=\frac12(\partial_2)^2f
\end{array}$$
to discover that $\Rho_{33}=-\frac14\big((\partial_1)^2+(\partial_2)^2\big)f$
with all other components zero and hence that
$$V^{ab}{}_c=\epsilon^{dea}W_{de}{}^b{}_c
=\epsilon^{dea}R_{de}{}^b{}_c-2\epsilon^{bea}\Rho_{ec}$$
is given by
%%%%%%%%%%%%%%%%%%%%%%%%%%%%%%%%%%%%
%% Private computation as follows %%
%%%%%%%%%%%%%%%%%%%%%%%%%%%%%%%%%%%%
% $$\textstyle V^{11}{}_3=\epsilon^{de1}R_{de}{}^1{}_3-2\epsilon^{1e1}\Rho_{e3}
% =2R_{23}{}^1{}_3=-\partial_1\partial_2f$$
% $$\textstyle V^{22}{}_3=\epsilon^{de2}R_{de}{}^2{}_3-2\epsilon^{2e2}\Rho_{e3}
% =2R_{31}{}^2{}_3=\partial_1\partial_2f$$
% $$\textstyle V^{33}{}_3=\epsilon^{de3}R_{de}{}^3{}_3-2\epsilon^{3e3}\Rho_{e3}
% =0$$
% $$\textstyle V^{12}{}_3=\epsilon^{de1}R_{de}{}^2{}_3-2\epsilon^{2e1}\Rho_{e3}
% =2R_{23}{}^2{}_3-2\Rho_{33}
% =\frac12\big((\partial_1)^2-(\partial_2)^2\big)f$$
% $$\textstyle V^{21}{}_3=\epsilon^{de2}R_{de}{}^1{}_3-2\epsilon^{1e2}\Rho_{e3}
% =2R_{31}{}^1{}_3+2\Rho_{33}=\frac12\big((\partial_1)^2-(\partial_2)^2\big)f$$
% $$\textstyle V^{13}{}_3=\epsilon^{de1}R_{de}{}^3{}_3-2\epsilon^{3e1}\Rho_{e3}
% =0$$
% $$\textstyle V^{23}{}_3=\epsilon^{de2}R_{de}{}^3{}_3-2\epsilon^{3e2}\Rho_{e3}
% =0$$
\begin{equation}\label{newtonian_curvature}V^{ab}{}_3=\left[\begin{array}{ccc}
-\partial_1\partial_2f
&\frac12\big((\partial_1)^2-(\partial_2)^2\big)f
&0\\
\frac12\big((\partial_1)^2-(\partial_2)^2\big)f
&\partial_1\partial_2f
&0\\
0&0&0
\end{array}\right]\end{equation}
with all other components zero. If $f(x^1,x^2)=\alpha x^1+\beta x^2+\gamma$,
then $V^{ab}{}_c$ vanishes and the structure is projectively flat. Otherwise
the constraint (\ref{curvature_constraint_revisited}) implies that
$\sigma^{c3}=0$ for all~$c$. Already, all solutions to
(\ref{metrisability_equation}) are degenerate so the structure is not 
metrisable. In fact, one can go on to check that 
$$\sigma^{ab}=\left[\begin{array}{ccc}
A&B&0\\ B&C&0\\ 0&0&0\end{array}\right]$$
is the general solution of (\ref{metrisability_equation}) for arbitrary 
constants $A,B,C$. As for the Egorov example, this $3$-dimensional space of 
solutions also precludes metrisability.
\end{proof}
\noindent{\em Yet Another Proof.} As for the Egorov example, the form of the
projective Weyl curvature (\ref{newtonian_curvature}) conflicts with
Lemma~\ref{keylemma} to provide yet another proof. This time $V^{ab}{}_c$ has
the form $X^{ab}Y_c$ for non-zero $X^{ab}$ and $Y_a$ (with $X^{ab}$ symmetric
but never simple). Lowering the indices with the purported metric gives
$V_{(abc)}=X_{(ab}Y_{c)}\not=0$, contrary to~(\ref{hook_symmetry}).
\hfill$\square$

Finally, we remark that, although the Weyl tensor $V^{ab}{}_c$ for the
Newtonian structure does not have the form required by Lemma~\ref{keylemma}, it
can still be written as
$$V^{ab}{}_c=2R^{d(a}g^{b)e}\epsilon_{edc},\quad\mbox{where}\enskip
R^{ab}=-\frac12\left[\begin{array}{ccc}
\partial_1\partial_1f&\partial_1\partial_2f&0\\ 
\partial_2\partial_1f&\partial_2\partial_2f&0\\ 0&0&0\end{array}\right]
\enskip\mbox{and}\enskip 
g^{ab}=\left[\begin{array}{ccc}1&0&0\\ 0&1&0\\ 0&0&0\end{array}\right],$$
which is of the required form save for $g^{ab}$ being degenerate.

\subsection{A Weyl metrisable but not metrisable projective structure}
Recall that a Weyl structure on $M$ consists of a conformal structure
$[g]$ together with a torsion-free connection $D$ that is compatible
with the conformal structure in a sense that
\[
Dg=\omega\otimes g
\]
for $g\in[g]$ and some $1$-form $\omega$, this compatibility condition being 
invariant under the transformation
\begin{equation}\label{weyl_freedom}
g\mapsto\hat{g}\equiv\Theta^2 g, \;\;\omega\mapsto \omega+2d(\ln{\Theta}),
\end{equation}
where $\Theta$ is a non-zero function on $M$. For any metric $g$ in the 
conformal class, the $1$-form $\omega$ determines the connection~$D$.

Consider a Lorentizian Weyl structure on the three-dimensional Heisenberg 
group~\cite{p_t}
\[
g=(dx^1)^2-(dx^2)^2+(dx^3-x^1dx^2)^2, \quad \omega=2(dx^3-x^1dx^2).
\]
This Weyl structure is Einstein--Weyl: the symmetrised Ricci tensor
of $D$ is proportional to $g$. Let $[D]$ be the projective structure
defined by $D$.

We find that the obstruction ${Q_{ab}}^c$ from Theorem~\ref{theorem_3} does not
vanish. For example, ${Q_{22}}^1=x^1$. The determinant $T$ from 
Theorem~\ref{theorem_1} also
does not vanish. A convenient way to present $T$ is to regard it as a ternary
sextic. Setting $X_{a}=(X, Y, Z)$ we find
\[
X_aX_bX_cX_dX_eX_fT^{abcdef}=Z^2(X+Y+Zx^1)^2(X-Y-Zx^1)^2
\]
up to a non-zero multiplicative constant.

{From} either of these obstructions, we therefore conclude that the projective
structure $[D]$ is not metrisable. It is nevertheless Weyl metrisable by
construction. In dimension two all projective structures are locally Weyl
metrisable \cite{mettler}. We expect this not to be the case in dimension
three, where, up to diffeomorphism, a general real-analytic projective
structure depends on $12$ arbitrary functions of $3$ variables, but a Weyl
structure only depends on $5$ such functions. Characterising projective
connections that are Weyl metrisable is an interesting open problem, which we
do not pursue here. In general, we also do not know which Weyl metrisable
structures are genuinely metrisable. If the Einstein--Weyl equations hold,
however, then we have a satisfactory answer as follows.

\subsection{Einstein-Weyl projective structures}
Consider a $3$-dimensional Weyl structure $(D, [g])$ as outlined at the 
beginning of the previous section.
% 
% represented by a metric $g_{ab}$ and
% a one-form $\omega_a$ subject to 
% \[
% g\mapsto \Omega^2 g, \;\;\omega\mapsto \omega+2d(\ln{\Omega}).
% \]
% A tensor object $T$ which transforms like $T\mapsto\Omega^m T$ when
% $g\mapsto \Omega^2 g$ is said to be conformally invariant of weight $m$.
% Therefore tensors
% \[
% g_{ab},\quad g^{ab}, \quad \mbox{and}\quad  \epsilon_{abc} 
% \;\mbox{(volume form)}
% \]
% have weights $2, -2$ and $3$. Then $\epsilon^{abc}$ has weight 
% $-3=3+3\times(-2)$   
% and so on. Note that the notions of conformal weights and projective weights
% are different.
In general, the Ricci tensor of $D$ contains both symmetric and skew parts, the
latter being proportional to $\nabla_{[a}\omega_{b]}$. The $2$-form 
$F_{ab}\equiv\nabla_{[a}\omega_{b]}$ is an invariant of the Weyl structure, 
often called the {\em Faraday form\/}. The Weyl structure is
called {\em Einstein--Weyl\/} if 
\be
\label{ew_eq}
\Phi_{ab}=0, 
\ee 
where $\Phi_{ab}$ is the symmetrised trace-free part of the
Ricci tensor of~$D$ (noting that removing the trace of a symmetric tensor
depends only on the conformal class~$[g]$).
\begin{theo}
Let $(D, [g])$ be an Einstein--Weyl structure in dimension 3, and  
let $[D]$ be the projective structure defined by $D$. 
Then $[D]$ is metrisable if and only if its Faraday form $F_{ab}$ vanishes. 
\end{theo}
\begin{proof}
% Let $W_{ab}{}^c{}_d$ be the projective Weyl curvature of $D$, and let
% $\Phi_{ab}$ be symmetrised Ricci tensor of $D$ with its trace removed.
% Set
% \[
% {V^{ab}}_c=\epsilon^{dea}W_{de}{}^b{}_c, \quad Q_{ab}{}^c=\epsilon_{pq(a}V^{pr}{}_{b)}V^{qc}{}_r.
% \]
% Note that $V$ and $Q$ are of conformal weights $-3$.
Let $V^{ab}{}_c$ be as usual~(\ref{V_from_W}). 
A straightforward but cumbersome calculation yields
% \be\label{V_weyl}
\[\textstyle
{V^{ab}}_c=2 \Phi^{d(a}{\epsilon^{b)}}_{dc}-\frac{1}{2} 
\delta_c{}^{(a}f^{b)}+g^{ab}f_c,\]
%={(V_R)^{ab}}_c+{(V_F)^{ab}}_c,
% \ee
% where $f^a$ is a co-vector of conformal weight $-3$ defined by
where $f^a$ is the vector field defined by
\[
F_{ab}=\epsilon_{abc}f^c,
\]
and indices are lowered and raised by any representative metric $g_{ab}$ 
from the conformal class and its inverse~$g^{ab}$.
An even more cumbersome calculation gives
% \be\label{Q_weyl}
\[{Q_{ab}}^c=
\Phi_{ab}f^c+2{\Phi^{c}}_{(a}f_{b)}-2{\delta^c}_{(a}{\Phi^d}_{b)}f_d
+2g_{ab}\Phi^{cd}f_d+f_{(a}{\epsilon_{b)d}}^c f^d,\]
% \ee
%\[
%Q=V_R V_R+V_R V_F+V_FV_R +V_F V_F,
%\]
%where
%\begin{eqnarray*}
%V_R V_R&=&0,\\ 
%V_F V_R&=&3 f_{(b}{R^{c}}_{a)}, \\
%V_R V_F&=&
%R_{ab}f^c-\frac{1}{2}{R^{c}}_{(a}f_{b)}
%+\frac{1}{4}{\delta^c}_{(a}{R^d}_{b)}f_d+\frac{3}{4}g_{ab}R^{cd}f_d\\
%\frac{1}{4}R_{ab}f^c+\frac{1}{4}{R^{c}}_{(a}f_{b)}
%+{\delta^c}_{(a}{R^d}_{b)}f_d+
%\frac{3}{4}{\epsilon^c}_{p(b}\epsilon_{a)dq}R^{pd} f^q
%V_F V_F&=&-f_{(b}{\epsilon_{a)d}}^c f^d.
%\end{eqnarray*}
where $Q_{ab}{}^c$ is our usual quadratic obstruction (\ref{this_is_Q}) to
metrisability. If the Einstein--Weyl equations (\ref{ew_eq}) hold then
$Q_{ab}{}^c$ vanishes if and only if $f^a=0$, which happens if and only if
$\omega$ is locally a gradient. But in this case the Weyl connection is the
Levi-Civita connection of a conformally rescaled metric~$\hat{g}$, so it is
metrisable (as it is metric).\end{proof}

In fact, if the Faraday form vanishes, and locally we choose a metric
connection in the projective class according to (\ref{weyl_freedom}), then the
Einstein--Weyl equations (\ref{ew_eq}) revert to the Einstein equations. Since
we are in $3$ dimensions, the Einstein equations imply that the metric is
constant curvature. Therefore, the only metrisable Einstein--Weyl structures
in $3$ dimensions are projectively flat. 

\section{Path geometries and systems of ODEs}
\label{section_4}
A convenient way to exhibit examples of projective structures on 
$U^{\mathrm{open}}\subseteq{\mathbb{R}}^3$ is to use an
equivalent definition of a path geometry in $3$ dimensions as an equivalence
class of systems of two second order ordinary differential equations
~\cite{CDT13} 
\be\label{system}
y''=F(x,y,z,y',z'), \quad z''=G(x,y,z,y',z'),
\ee
where two systems are regarded as equivalent if they can be mapped into each
other by a change of dependent and independent variables 
$(x,y,z)\mapsto(\ov{x}(x,y,z),\ov{y}(x,y,z),\ov{z}(x,y,z))$.
An integral curve of (\ref{system}) is, for sufficiently regular functions $F,
G$, specified uniquely by a point and a direction in $U$.

It is relatively straightforward to characterise 2nd order systems
(\ref{system}) that give rise to projective path geometries \cite{casey,Fels}:
set $y^i=(y, z)$, $p^i=(y', z')$ and $F^i=(F, G)$, where the indices $i, j, k,
\dots $ take values $2, 3$. The necessary and sufficient conditions for the
integral curves of (\ref{system}) to be unparametrised geodesics of a
torsion-free connection on $TU$ are \cite{Fels} 
\be
\label{fels}
{S^{i}}_{(jkl)}=0, \quad\mbox{where}\quad
{S^i}_{jkl}=\frac{\p^3 F^i}{\p p^j\p p^k\p p^l}
-\frac{3}{4}\frac{\p^3 F^m}{\p p^m\p p^j\p p^k}\delta_l{}^i.
\ee

To establish this result it is enough to consider the geodesic equations for a
given $\nabla$, and eliminate the affine parameter $s$ between the three
equations
\[
\frac{d^2 x^a}{d s^2}+\Gamma_{bc}{}^a\frac{d x^b}{d s}\frac{d x^c}{d s} =0,
\]
where $x^a=(x, y, z)$. This yields (\ref{system}), with
\be
\label{fels_ode}
F^i=A_{jk}p^ip^jp^k+{B^i}_{jk}p^jp^k+{C^i}_jp^j+D^i,
\ee
where 
\[
A_{ij}=\Gamma_{ij}{}^1, \quad 
{B^i}_{jk}=2\Gamma_{1(j}{}^1\delta_{k)}{}^i-\Gamma_{jk}{}^i, \quad
{C^i}_j=\Gamma_{11}{}^1\delta_j{}^i-2\Gamma_{1j}{}^i, \quad 
D^i=-\Gamma_{11}{}^i.
\]
Note that the expressions for $A, B, C, D$ are invariant under
(\ref{equivalence}). Conversely, imposing (\ref{fels}) on system (\ref{system})
yields (\ref{fels_ode}) as in~\cite{casey}. For example the Egorov projective
structure (\ref{egorov_p}) corresponds to a system
\[
y''=2y(y')^2z', \quad z''=2yy'(z')^2.
\]
Expressing any of the projective invariants in this article, such
as~(\ref{MN}), in terms of $F$, $G$, and their derivatives gives 
point invariants of system (\ref{system}).

\end{document}